%% file: MAIN_ArXiv.tex
\documentclass{article}
\usepackage[utf8]{inputenc}
\usepackage{authblk}

\title{Disconnected cuts in 4-connected planar graphs}
\author{Brandon Du Preez}
\affil{University of Cape Town}
\affil{brandon.dupreez@uct.ac.za}
\date{13 December 2023}

\usepackage{graphicx}
\usepackage[a4paper,margin=1.5in,footskip=0.25in]{geometry}
\usepackage {amsmath, amssymb, amsthm, amsfonts}
\usepackage{xcolor}
\usepackage{lineno}
\usepackage{tikz}
\usepackage{bm}
\usepackage{changepage}
\usepackage{float}
\usepackage[algoruled, vlined]{algorithm2e}
\usepackage{parskip}
\usetikzlibrary{arrows,decorations.pathmorphing,backgrounds,positioning,fit,petri, arrows.meta, matrix, calc, decorations.markings, shapes}
\tikzset{>=latex}

\SetEndCharOfAlgoLine{}
\LinesNumbered
\SetNlSty{textbf}{}{:}
\SetKwFunction{Subset}{Subset}
\SetKwFunction{RemoveChords}{RemoveChords}
\SetKwFunction{Embed}{Embed}
\SetKwFunction{ListLargeFaces}{ListLargeFaces}
\SetKwFunction{FaceIntersection}{FaceIntersection}
\SetKwFunction{Lookup}{Lookup}
\SetKwFunction{MengerPaths}{MengerPaths}
\SetKwFunction{TruncatePath}{TruncatePath}
\SetKwFunction{Append}{Append}
\SetKwFunction{PathSkipper}{PathSkipper}
\SetKwFunction{In}{in}
\SetKw{Return}{return}
\SetKw{Initialise}{initialise}
\SetKw{End}{End}

\renewcommand{\int}{\ensuremath{\textup{Int}}}
\newcommand{\ext}{\ensuremath{\textup{Ext}}}
\newcommand{\adj}{\ensuremath{\textup{Adj}}}
\newcommand{\n}{\ensuremath{\textup{N}}}

\newcommand{\ins}[1]{{#1}^{\bm \circ}}
\newcommand{\DC}{\textsc{Disconnected Cut}}
\newcommand{\MDC}{\textsc{Minimal Disconnected Cut}}

\begin{document}
	\newtheorem{theorem}{Theorem}
	\newtheorem{thm}[theorem]{Theorem}
	\newtheorem{lem}[theorem]{Lemma}
	\newtheorem{cor}[theorem]{Corollary}
	\newtheorem{obs}[theorem]{Observation}
	\newtheorem{prop}[theorem]{Proposition}
	\newtheorem{conj}[theorem]{Conjecture}
	\newtheorem{dfn}[theorem]{Definition}
	\newtheorem{rem}[theorem]{Remark}
	\newtheorem{prob}{Problem}
	
	\maketitle
	
	\begin{abstract}
		\noindent Let $G=(V,E)$ be a connected graph. A subset $S\subset V$ is a cut of $G$ if $G-S$ is disconnected. A near triangulation is a 2-connected plane graph that has at most one face that is not a triangle. In this paper, we explore minimal cuts of 4-connected planar graphs. Our main result is that every minimal cut of a 4-connected planar graph $G$ is connected if and only if $G$ is a near-triangulation. We use this result to sketch a linear-time algorithm for finding a disconnected cut of a 4-connected planar graph.
	\end{abstract}

\section{Introduction}

Let $G=(V,E)$ be a connected graph.
A \textbf{cut} of $G$ is a set $S\subset V$ such that $G-S$ is disconnected.
The cut $S$ is \textbf{minimal} if $G-T$ is connected for each proper subset $T\subset S$.
This is equivalent to the (easier to check) property that $G-(S-v)$ is connected for every $v\in S$.
The cut $S$ is connected (disconnected) if the induced graph $G[S]$ is connected (disconnected), and it is stable if $G[S]$ has no edges.

The decision problem \DC{} asks whether a given connected graph $G=(V,E)$ has a disconnected cut.
This problem is well-studied --- in part, because it is equivalent to a number of other problems related to graph homomorphisms, graph covers, and graph partitions.
Phrased as a graph homorphism problem, \DC{} is equivalent to asking whether there exists a vertex-surjective homomorphism from $G$ onto $\mathcal{C}_4$, where $\mathcal{C}_4$ is the graph formed by adding a loop to each vertex of the cycle $C_4$ \cite{cov_few_bip_subgraphs}.
\DC{} is also equivalent to the covering problem of determining whether the complement graph $\overline{G}$ contains two complete bipartite subgraphs $H_1$ and $H_2$ such that $V(H_1)\cup V(H_2) = V$ \cite{cov_few_bip_subgraphs}.
For details on how \DC{} is equivalent to two $H$-partition problems (the $2K_2$ and $2S_2$ problems), see \cite{disc_cuts_and_seps, param_cut_sets}.

As shown in \cite{complexity_disc_cut}, \DC{} is NP-complete for graphs of diameter 2 (the problem is trivial when the diameter is 1, and it's easily shown that every diameter 3 graph has a disconnected cut).
However, the problem is fixed parameter tractable on graphs of bounded genus --- where the parameter involved is the number of components in the cut \cite{param_cut_sets}.
Further, for many graph classes, \DC{} is known to be tractable.
Examples of such classes include triangle-free graphs, graphs with a dominating edge, and graphs with bounded maximum degree \cite{cov_few_bip_subgraphs}.
In \cite{param_cut_sets}, it is shown that \DC{} can be solved in polynomial time for chordal graphs, and for any minor-closed family of graphs that excludes any apex graph.
As a notable corollary, \DC{} is polynomial time solveable for planar graphs.
By developing a decomposition theorem for claw-free graphs of diameter 2, \cite{disc_cuts_claw_free} shows that \DC{} is polynomial time solveable for claw-free graphs. 
Further, \cite{disc_cuts_claw_free} shows that the problem is also polynomial time solveable for $H$-free graphs, where $H$ is any 4-vertex graph other than $K_4$.
Both \cite{complexity_disc_cut} and \cite{disc_cuts_claw_free} contain extensive lists of related results.

In some cases, the cuts of a graph are not as useful / structured as the \textit{minimal} cuts (the not-necessarily minimal cuts are also far more numerous). 
This is especially true for 3-connected plane graphs, in which minimal cuts correspond to Jordan Curves that separate the graph \cite{listing_min_seps_3conn_pg}.
Thus the focus of this paper, the \MDC{} decision problem, was introduced in \cite{disc_cuts_and_seps}:

\fbox{
\parbox{370pt}{
\MDC{}\\
\textbf{Input:} A connected graph $G$.\\
\textbf{Question:} Does $G$ have a minimal cut that is disconnected?
}}

Note that we are interested in `inclusion-minimal cuts that are disconnected' and \textit{not} `minimal elements of the set of disconnected cuts'.
Like \DC{}, \MDC{} is NP-Complete \cite{disc_cuts_and_seps}.
For graphs of diameter 2, these two problems are equivalent, but \MDC{} remains NP-Complete when the input graph has diameter 3.
As shown in \cite{min_disc_cuts_in_pgs}, \MDC{} is NP-Complete for planar graphs (specifically, those with connectivity 2), but can be solved in $\mathcal{O}(n^3)$ time for 3-connected planar graphs.

In this paper, we show that \MDC{} can be solved in linear time on 4-connected planar graphs. 
Moreover, if a minimal disconnected cut exists, one can be found in linear time. 
We will need a number of results and technical lemmas on minimal cuts in planar graphs --- most of which can be found in \cite{chordal_embeddings_pgs}, \cite{listing_min_seps_3conn_pg} and \cite{min_disc_cuts_in_pgs}.
For other recent work on minimal cuts in graphs and planar graphs, see \cite{struct_min_vertex_cuts} and \cite{listing_min_seps_lintime}.

\section{Overview}
In Section \ref{sec:prelim}, we lay out some literature results that will either be needed for the rest of the paper, or that provide intuition for minimal cuts in highly connected planar graphs.
In Section \ref{sec:nhoods}, we describe the local structure (faces and vertex neighbourhoods) of 3-connected and 4-connected planar graphs.
In section \ref{sec:disc_sep} we prove that a 4-connected planar graph is cleavable if and only if it has at most one face that is not bounded by a triangle.
In section \ref{sec:auxillary}, we investigate properties of the auxiliary graph (a tool we make use of throughout the paper).
We then present some further results on components of minimal cuts of planar graphs in Section \ref{sec:components} --- with a focus on results relevant to the \textsc{Stable Cut} problem (the problem of finding a cut that is an independent set).
In Section \ref{sec:algorithm}, we show that a minimal disconnected cut of a 4-connected planar graph can be found in linear time.

\section{Further definitions}
\label{sec:defins}

Let $G=(V,E)$ be a graph. 
Suppose $P = v_1, \dots, v_k$ is a path in $G$, then $\ins{P} = P - \{v_1, v_k\}$. If $H$ is a subgraph of $G$, then $P$ and $H$ are \textbf{internally disjoint} if $\ins{P}\cap H = \emptyset$.
We say that a set $T\subseteq V$ of vertices is connected if the induced graph $G[T]$ is connected.
If $A$ and $B$ are two disjoint subsets of $V$, then $S$ is an $\bm{A-B}$ \textbf{separator} if $A$ and $B$ are in different components of $G-S$. 
The set $S$ is a \textbf{minimal} $\bm{A-B}$ \textbf{separator} if $S$ is an $A-B$ separator, and for each $T\subset S$, $A$ and $B$ are in the same component of $G-T$.
For two vertices $u$ and $v$, (minimal) $u-v$ separators are defined analogously. 
Every separator is a cut, however a minimal $A-B$ separator need not be a minimal cut.
We say $G$ is \textbf{cleavable} if every minimal cut of $G$ is connected.

Let $G=(V,E)$ be a plane graph in which every face is bounded by a cycle (i.e., $G$ is 2-connected). 
A cycle in $G$ is called a \textbf{large face} if it bounds a face, and has length 4 or greater.  
We say that $G$ is a \textbf{triangulation} / \textbf{maximal plane graph} if every face is bounded by a triangle. 
Further, $G$ is a \textbf{near triangulation} if every face is bounded by a cycle, and at most one of the faces is bounded by a cycle of length greater than 3 (i.e., a near triangulation is a 2-connected plane graph with at most one large face).
We will often refer to a face and its bounding cycle interchangeably. 
Where confusion may arise, denote by $C_f$ the cycle bounding face $f$, and $f_C$ the face bounded by cycle $C$.
Every cycle $C$ of $G$ induces a jordan curve in the plane --- and thus partitions the plane into the interior of $C$, the exterior of $C$, and the curve $C$ itself.

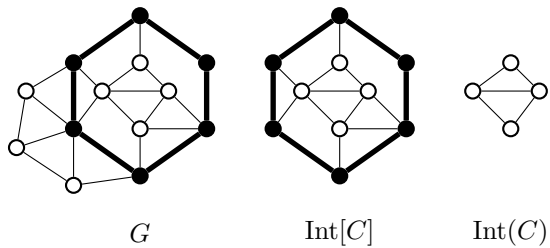
\begin{figure}[h]
	\begin{center}
		\input{jordan_separating}
		\caption{Left to right: a Jordan-separating cycle $C$ (bold) in a plane graph $G$; the subgraph $\int[C]$ of $G$; the subgraph $\int(C)$ of $G$.}
	\end{center}
	\label{fig:jordan_separating}
\end{figure}

Denote by $\int(C)$ ($\ext(C)$) the subgraph of all the edges and vertices lying strictly in the interior (exterior) of $C$ --- and $\int[C]$ ($\ext[C]$) the subgraph of all edges and vertices lying in the closed region consisting of $C$ and its interior (exterior).
We say that a cycle $C$ is a \textbf{Jordan separating cycle} if both $\int(C)$ and $\ext(C)$ each contain at least one vertex (see Figure \ref{fig:jordan_separating}).

Suppose $G=(V,E)$ is a plane graph in which every face is bounded by a cycle, and let $\Omega$ be the set of large faces of $G$. 
To form the \textbf{auxiliary graph} $\bm{G^\nabla}$ of $G$, add a vertex $v_W$ in the face $f_W$ for each large face $W\in \Omega$, and make $v_W$ adjacent to every vertex of $W$ (for example, see Figure \ref{fig:near_triangulation}). 
Note that $G^\nabla$ is a triangulation.

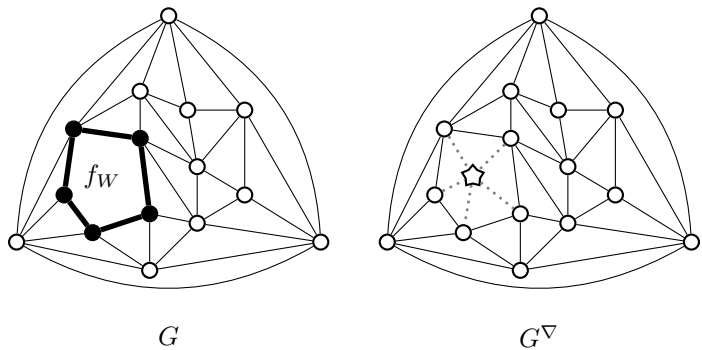
\begin{figure}[h]
	\begin{center}
		\input{near_triangulation}
		\caption{Left: a near triangulation $G$ with a single large face $W$ (bold). Right: the augmented graph $G^\nabla$ --- the edges in $G^\nabla$ but not $G$ are bold, and the vertex is star-shaped.}
	\end{center}
	\label{fig:near_triangulation}
\end{figure}

\section{Preliminaries}
\label{sec:prelim}

We begin with some well known technical lemmas.
Throughout the paper, note that a plane graph is 2-connected if and only if every face is bounded by a cycle \textup{\cite{bondy_murty_gt}}.
A variation of Proposition \ref{prop:min_sep_is_adjacent} can be found in \textup{\cite{diestel_gt}}.

\begin{prop}\textup{\cite{diestel_gt}}
	Let $S$ be a cut of a graph $G$. Then $S$ is minimal if and only if every vertex of $S$ is adjacent to every component of $G-S$.
	\label{prop:min_sep_is_adjacent}
\end{prop}

\begin{prop}\textup{\cite{chordal_embeddings_pgs}}
	Let $G$ be a 3-connected planar graph. 
	If $S$ is a minimal separator of $G$, then $G-S$ has exactly two components.
	\label{prop:min_sep_3conn_2_components}
\end{prop}

As a consequence of Propositions \ref{prop:min_sep_is_adjacent} and \ref{prop:min_sep_3conn_2_components}, we easily derive Proposition \ref{prop:min_ABsep_is_min_sep}.

\begin{prop}\textup{\cite{listing_min_seps_3conn_pg}}
	Let $G$ be a 3-connected planar graph, and $A$, $B$ two connected sets of vertices of $G$.
	If $S$ is a minimal $A-B$ separator, then it is also a minimal cut.
	\label{prop:min_ABsep_is_min_sep}
\end{prop}

The following characterization of minimal cuts in triangulations is well known (see, eg, \cite{k_path_ham_baybars}), for an elementary proof see \cite{dist_pg_thesis}.

\begin{prop}
	Let $G$ be a maximal planar graph, and suppose $A, B$ are disjoint, connected sets of vertices in $G$.
	Every minimal cut of $G$ --- and every minimal $A-B$ separator --- induces a chordless Jordan separating cycle.
	Conversely, if $C$ is a chordless Jordan separating cycle of $G$, then $V(C)$ is a minimal cut.
	\label{prop:min_sep_mpg_is_cycle}
	\label{prop:chordless_cycle_is_min_sep}
\end{prop}

The characterization in Proposition \ref{prop:min_sep_mpg_is_cycle} can be easily used to describe the minimal cuts of near-triangulations.

\begin{thm}\textup{\cite{on_mpgs}}
	Suppose $G$ is a near-triangulation, $S$ is a minimal cut of $G$, and $W$ is the single large face of $G$. 
	Then we have one of the following:
	\begin{itemize}
		\item $G[S]$ is a chordless cycle such that $S\cap V(W)$ induces an empty graph, a single vertex, or a single edge, or
		\item $G[S]$ is a path $P = v_0, v_1, \dots, v_k$ such that $P\cap W = \{v_0, v_k\}$.
	\end{itemize}
	\label{thm:min_sep_of_near_triangulation}
\end{thm}

\begin{figure}[h]
	\begin{center}
		\input{min_seps_near_mpg}
		\caption{Three different minimal cuts (bold) of a near-triangulation.}
	\end{center}
	\label{fig:min_seps_near_mpg}
\end{figure}
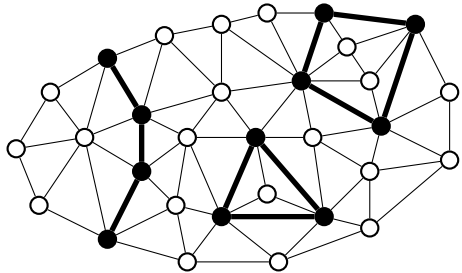

\begin{cor}
	Every near-triangulation is cleavable.
	\label{cor:near_triangulation_cleavable}
\end{cor}

A full proof of Theorem \ref{thm:min_sep_of_near_triangulation} will appear in \cite{on_mpgs} --- but the reader can easily convince themselves of it by extending a cut of $G$ to a cut of the auxiliary graph $G^\nabla$ and applying Proposition \ref{prop:min_sep_mpg_is_cycle}.
Note that these results can be generalized further: it is proven in \textup{\cite{min_disc_cuts_in_pgs}} that every minimal separator of a planar graph induces either a cycle, or a disjoint union of paths.

\section{Neighborhoods and faces}
\label{sec:nhoods}

We'll need some simple, well known facts about 3-connected planar graphs.
Consider a face-cycle $C = v_0, \dots, v_k$ of a plane graph, and suppose $C$ has a chord $v_iv_j$.
The set $\{v_i,v_j\}$ is a $v_s-v_t$ separator for any $s$ and $t$ such that $v_i, v_s, v_j, v_t$ appear in that cyclic order around $C$, yielding Remark \ref{rem:3conn_face_cycle_chordless} below.

\begin{rem}
	Let $G$ be a plane graph, and $C$ a face cycle in $G$.
	If $G$ is 3-connected, then $C$ is chordless.
	\label{rem:3conn_face_cycle_chordless}
\end{rem}

Suppose $W$ and $X$ are two large faces of a plane graph such that $X\cap W$ contains a pair $u,v$ of non-adjacent vertices. 
Then we can construct a Jordan curve $J$ starting at $u$, traversing through the face $f_X$ to $v$, then passing through $f_W$ back to $u$. 
Since $u$ and $v$ are not adjacent in $W$, the curve $J$ will contain vertices of $W$ in both its interior and exterior --- so $\{u,v\}$ will be a cut. 
Thus we have Remark \ref{rem:wells_intersect_at_k1_or_k2}.

\begin{rem}
	Let $G$ be a 3-connected plane graph, and suppose $X$ and $W$ are two face cycles in $G$.
	Then $X\cap W$ is either a single edge lying on both $f_X$ and $f_W$, or it is a single vertex.
	\label{rem:wells_intersect_at_k1_or_k2}
\end{rem}

The following is a slight strengthening of the well known fact that no face-cycle disconnects a 3-connected planar graph (see, for example, \cite{bondy_murty_gt}).

\begin{lem}
	Let $G$ be a 3-connected plane graph, and $C$ a face-cycle of $G$.
	For every pair $u,v$ of vertices in $G$, there exists a $u-v$ path that is internally disjoint from $C$.
	\label{lem:3conn_path_avoids_face}
\end{lem}

\begin{proof}
	Let $f$ be the face bounded by $C$.
	By Menger's theorem, there exist three internally disjoint $u-v$ paths $P,Q$ and $R$.
	These paths create three cycles $P\cup Q$, $Q\cup R$ and $P\cup R$.
	Up to relabelling of the paths and regions, the face $f$ lies in the interior of $Q\cup R$, and $P$ is the desired path.
\end{proof}

We'll need the following simple lemma. 

\begin{lem}
	Let $G$ be a 4-connected planar graph, and $u$ a vertex of $G$.
	Then $\n(u)$ is a minimal cut of $G$.
	\label{lem:neighbourhood_minn_sep_4conn}
\end{lem}

\begin{proof}
	It's clear that $\n(u)$ is a cut, so it remains to prove minimality. 
	Per Proposition \ref{prop:min_sep_is_adjacent}, it suffices to prove that $G-\n[u]$ is connected.
	Since $G$ is 4-connected, $G-u$ is 3-connected, so every face of $G-u$ is bounded by a cycle.
	In particular, the face of $G-u$ containing $u$ is bounded by a cycle containing every vertex of $\n(u)$.
	Thus $G-\n[u]$ is connected, by Lemma \ref{lem:3conn_path_avoids_face}.
\end{proof}

\begin{cor}
	Suppose $G$ is a 4-connected, cleavable plane graph.
	For every vertex $u$ of $G$, the neighbourhood $\n(u)$ is connected.
	\label{cor:cleavable_nhood_is_connected}
\end{cor}

The connectivity requirement of Lemma \ref{lem:neighbourhood_minn_sep_4conn} cannot be improved. 
Consider the graph $G = P_5 + K_2$ formed by making every vertex of a $P_5$ adjacent to every vertex of a $K_2$.
This is a 3-connected planar graph, and if $u$ is a vertex of degree 4, then $\n(u)$ is not a minimal cut of $G$.

The next Proposition shows that in a 4-connected cleavable plane graph, no two large faces intersect.

\begin{prop}
	Let $G$ be a 4-connected plane graph, and let $W, X$ be two large faces of $G$.
	If every neighbourhood of $G$ is connected, then $W\cap X = \emptyset$.
	\label{prop:wells_dont_intersect}
\end{prop}

\begin{proof}
	Assume to the contrary that $W\cap X \neq \emptyset$.
	Per remark \ref{rem:wells_intersect_at_k1_or_k2}, there are two cases to consider. 
	In both cases, we show that some vertex of $G$ has a disconnected neighbourhood.
	
	\textit{Case 1:} $W\cap X$ consists of two adjacent vertices $u$ and $v$ such that $uv$ lies on the boundary of both $f_W$ and $f_X$.
	Denote by $s$ and $t$ the vertices such that $s,u,v,t$ is a path in $W$, and $y,z$ the vertices such that $y,u,v,z$ is a path in $X$.
	Observe that graph $G-u$ is 3-connected. 
	Let $f$ be the face of $G-u$ containing $u$, and note that this face is bounded by a cycle $C$.
	Every vertex of $\n(u)$ lies on $C$, and the neighbours of $v$ on $C$ are $t$ and $z$ (see Figure \ref{fig:faces_meet_on_edge}). 

	By remark \ref{rem:3conn_face_cycle_chordless}, the neighbour $v$ is not adjacent to any other vertex of $\n(u)$.
	Thus $\n(u)$ is disconnected.
	
	\begin{figure}[h]
		\begin{center}
			\input{faces_meet_on_edge}
			\caption{Left: In Case 1, two large faces share an edge $uv$. Right: In Case 2, the two large faces share only the vertex $u$. In each case, the edges of the cycle $C$, and the neighbors of $u$, are bolded.}
		\end{center}
		\label{fig:faces_meet_on_edge}
	\end{figure}
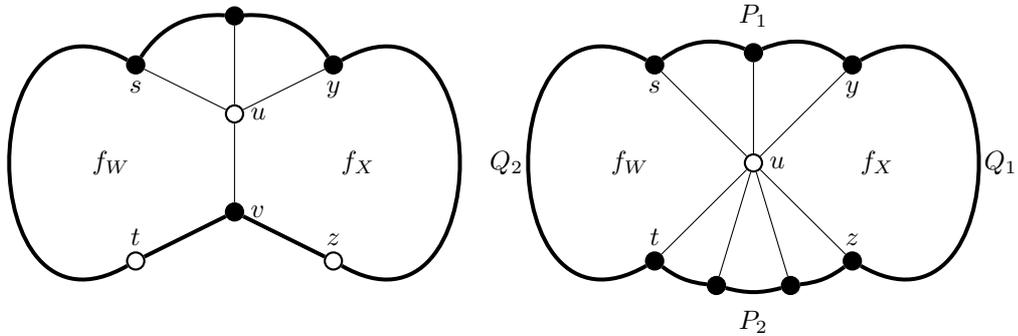
	
	\textit{Case 2:} $W\cap X$ is the single vertex $u$.
	Let $s,t$ be the vertices such that $s,u,t$ is a path in $W$ and $y,z$ vertices such that $y,u,z$ is a path in $X$.
	As in Case 1, the face of $G-u$ containing $u$ is bounded by a cycle $C$ containing every vertex of $\n(u)$.
	Up to swapping the labels of $s,t$, and swapping the labels on $y,z$, the vertices $s,y,z,t$ appear in that cyclic order on $C$.
	Let $P_1$ be the $s-y$ path in $C$, $Q_1$ the $y-z$ path, $P_2$ the $z-t$ path and $Q_2$ the $t-s$ path (see Figure \ref{fig:faces_meet_on_edge}). 
	
	Note that $Q_1\leq X$ and $Q_2\leq W$ --- and that $\n(u) \subseteq V(P_1)\cup V(P_2)$.
	To show that $\n(u)$ is disconnected, it suffices to prove that there is no edge between $P_1$ and $P_2$.
	Since $W$ and $Z$ have length at least 4, both $Q_1$ and $Q_2$ have length at least 2. 
	Thus there is no edge between $P_1$ and $P_2$ in $C$.
	Further, per remark \ref{rem:3conn_face_cycle_chordless}, there is no chord of $C$ from a vertex of $P_1$ to a vertex of $P_2$.
	Therefore $\n(u)$ is disconnected. 
\end{proof}
	
Corollary \ref{cor:well_meet_is_min_disc_cut} below follows immediately from the proof of Proposition \ref{prop:wells_dont_intersect}.
	
\begin{cor}
	If a vertex $v$ lies in two large faces of a 4-connected planar graph $G$, then $\n(v)$ is a disconnected minimal cut.
	\label{cor:well_meet_is_min_disc_cut}
\end{cor}

\begin{cor}
	If a 4-connected plane graph has two intersecting large faces, then it is not cleavable.
	\label{cor:wells_meet_disconn_sep}
\end{cor}

Note that corollary \ref{cor:wells_meet_disconn_sep} does not hold for 3-connected plane graphs, as demonstrated by the graph in Figure \ref{fig:3conn_cleavable_inter_wells}.

\begin{figure}[h]
	\begin{center}
		\input{3conn_cleavable_inter_wells}
		\caption{A 3-connected, cleavable graph in which there are intersecting large faces of length 4. Note that the universal vertex (bold) is in every cut of the graph.}
	\end{center}
	\label{fig:3conn_cleavable_inter_wells}
\end{figure}
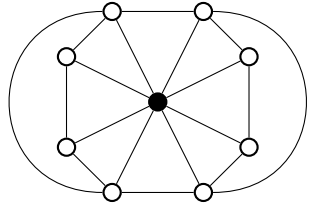

\section{Finding the disconnected cut}
\label{sec:disc_sep}

Let $G$ be a 2-connected plane graph, $G^\nabla$ its auxiliary graph, and $C$ a chordless cycle of $G^\nabla$.
For the following definitions, the cycle $W$ is a large face of $G$, and $w^\nabla$ is the vertex of $G^\nabla - G$ lying in the face $f_W$.
The following list of possibilities for the relationship between $W$ and $C$ is exhaustive, and pairwise mutually exclusive:
\begin{itemize}
	\item $W$ \textbf{touches} $C$ if $W\cap C$ is either a vertex, or a pair of vertices that are adjacent in both $W$ and $C$,
	\item $W$ is \textbf{inside} (\textbf{outside}) of $C$ if $W$ lies in $\int(C)$ ($\ext(C)$),
	\item $W$ \textbf{covers} $C$ if there exist vertices $x,y$ of $W$ such that $x, w^\nabla, y$ is a path in $C$, and $(C\cap W) - \{x,y\} = \emptyset$,
	\item $W$ \textbf{crosses} $C$ if it does not cover $C$, and there exist two vertices of $W\cap C$ that are not adjacent in $W$ (in this case, $W\cap C$ can contain many vertices).
\end{itemize}
Further, suppose $s,t$ are two non-adjacent vertices of $C$. 
Then the edges of $C$ are partitioned into two $s-t$ paths $P_1$ and $P_2$.
We say that the large face $W$ \textbf{splits} $s$ and $t$ in $C$ if $(W\cap P_1) - \{s,t\} \neq \emptyset$ and $(W\cap P_2) - \{s,t\} \neq \emptyset$.

\begin{figure}[h]
	\begin{center}
		\input{cross_split_touch}
		\caption{A cycle $C$ (bold outline grey vertices) with four large faces. From left to right: a large face splitting $s$ and $t$, a large face touching $C$, a large face crossing $C$ and a large face covering $C$.}
	\end{center}
	\label{fig:cross_split_touch}
\end{figure}
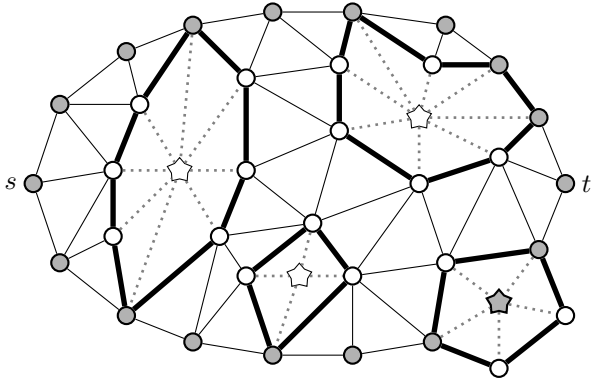

\begin{thm}
	Let $G$ be a 2-connected plane graph, $G^\nabla$ its auxiliary graph, and $C$ a chordless Jordan-separating cycle of $G^\nabla$.
	If no large face of $G$ crosses $C$, then $V(C)\cap V(G)$ is a minimal cut of $G$.
	\label{thm:auxillary_min_sep_restricts}
\end{thm}

\begin{proof}
	Let $S = V(C)\cap V(G)$. 
	By Proposition \ref{prop:min_sep_is_adjacent}, it suffices to prove both of the following:
	\begin{itemize}
		\item[\textbf{(1)}] both $\int(C)\cap G$ and $\ext(C)\cap G$ are connected, and
		\item[\textbf{(2)}] Every vertex of $S$ has a neighbour in $\int(C)\cap G$, and a neighbour in $\ext(C)\cap G$.
	\end{itemize} 

	\textit{Claim 1:} The subgraphs $\int(C)\cap G$ and $\ext(C)\cap G$ are connected.\\
	The set $V(C)$ is a minimal cut of $G^\nabla$ by Proposition \ref{prop:chordless_cycle_is_min_sep}.
	Thus $\int(C)$ and $\ext(C)$ are the connected components of $G^\nabla-C$ (note that $G^\nabla$ is 3-connected and use Proposition \ref{prop:min_sep_3conn_2_components}).
	In particular, $\int(C)$ is connected in $G^\nabla$.
	
	Let $u,v$ be two vertices of $\int(C)\cap G$.
	Since $\int(C)$ is connected in $G^\nabla$, there exists a $u-v$ path in $\int(C)$. 
	We prove that there must exist such a path containing only vertices of $G$.
	Assume to the contrary that every $u-v$ path in the interior of $C$ contains a vertex of $G^\nabla-G$.
	Among all such paths, let $P$ be a path containing the fewest vertices of $G^\nabla - G$.
	Let $w^\nabla$ be a vertex of $P$ in $G^\nabla-G$, and let $W$ denote the large face of $G$ that $w^\nabla$ lies inside.
	There exist vertices $x,y$ of $W$ such that $x,w^\nabla,y$ is a path in $P$.
	The large face $W$ is divided into two internally disjoint $x-y$ paths --- call them $Q_1$ and $Q_2$.
	By assumption, the large face $W$ does not cross $C$.
	Since $w^\nabla$ is in $\int(C)$, $W$ does not cover $C$.
	Thus $W$ is either inside $C$, or touches $C$.
	In either case, at least one of the paths $Q_i$ does not intersect $C$. 
	Thus we create a new $u-v$ path $P' = P[u,x]\cup Q_i \cup P[y,v]$.
	Note that $P'$ is contained in $\int(C)$, and contains one less vertex of $G^\nabla-G$ than $P$ (see Figure \ref{fig:path_detour}).
	This contradicts the minimality of $P$, proving that $\int(C)\cap G$ is connected.
	The proof that $\ext(C)\cap G$ is connected is similar.
	
	\begin{figure}[h]
		\begin{center}
			\input{path_detour}
			\caption{Modifying a path in $\int(C)$ to avoid the starred vertex of $G^\nabla-G$. The grey vertices lie on the cycle $C$.}
		\end{center}
		\label{fig:path_detour}
	\end{figure}
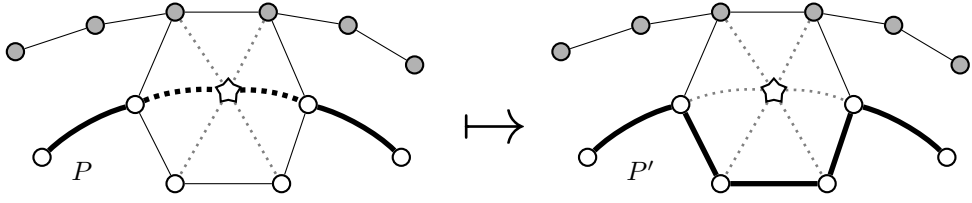
	
	\textit{Claim 2:} Every vertex of $S$ has a neighbour in $\int(C)\cap G$, and a neighbour in $\ext(C)\cap G$.\\
	Let $u$ be a vertex of $S$, and let $x$ and $y$ be the neighbours of $u$ in $C$. 
	Since $C$ is chordless, $x$ and $y$ are not adjacent.
	There are two cases.
	We show in both cases that $u$ has a neighbour in $\int(C)$ --- that $u$ has a neighbour in $\ext(C)$ follows similarly.
	
	\textit{Case 1:} Either $x$ or $y$ is a vertex of $G^\nabla-G$. 
	Assume without loss of generality that $x$ is not in $G$. 
	Then there exists a large face $W$ of $G$ covering $C$ such that $u$ is a vertex of $W$, and $x$ lies inside the large face.
	Since $C$ is chordless, $u$ has two neighbours $s,t$ in $W-C$.
	One of $s,t$ is in $\int(C)$, and the other in $\ext(C)$, so $u$ has a neighbour in $\int(C)$.
	
	\textit{Case 2:} Both $x$ and $y$ are in $G$.
	Let $f$ be the face of $G$ that is both incident with the edge $xu$ and lies in the interior of $C$.
	Let $W$ be the cycle bounding $f$, and note that there exists a vertex $s$ such that $x,u,s$ is a path in $W$.
	Since $C$ is chordless, and no large face crosses $C$, $s$ is not in $V(C)$.
	Thus $s\in \int(C)$, completing the proof.
\end{proof}

Let $G$ be a 2-connected plane graph, $G^\nabla$ it's auxiliary graph, and $C$ a chordless cycle of $G^\nabla$. 
Then $\bm{\zeta(C)}$ denotes the number of large faces of $G$ crossing $C$.

\begin{thm}[Crossing Removal Algorithm]
	Let $G$ be a 2-connected plane graph in which no two large faces intersect. Suppose it has auxiliary graph $G^\nabla$, and that $C$ a chordless cycle of $G^\nabla$ such that $\zeta(C) \geq 1$.
	Let $s,t$ be two non-adjacent vertices of $C$.
	If no large face of $G$ splits $s$ and $t$ in $C$, then there exists a chordless cycle $D$ of $G^\nabla$ with all of the following properties:
	\begin{itemize}
		\item[\textbf{(1)}] both $s$ and $t$ are vertices of $D$,
		\item[\textbf{(2)}] no large face of $G$ splits $s$ and $t$ in $D$,
		\item[\textbf{(3)}] $\zeta(D) < \zeta(C)$,
		\item[\textbf{(4)}] $D\cap G \leq C\cap G$,
		\item[\textbf{(5)}] Suppose $P_1$ and $P_2$ are two chordless $s-t$ paths such that $P_i - \{s,t\} \leq G$ ($i=1,2$). 
		If $P_1$ and $P_2$ are in different regions of $G^\nabla - C$, then they are in different regions of $G^\nabla - D$.
	\end{itemize}
	\label{thm:crossing_removal_algorithm}
\end{thm}

\begin{proof}
	Let $W$ be a large face crossing $C$, and let $w^\nabla$ be the vertex of $G^\nabla-G$ inside $W$.
	The vertices $s$ and $t$ divide $C$ into two internally disjoint $s-t$ paths $Q_1$ and $Q_2$.
	Since $W$ does not split $s$ and $t$, we assume without loss of generality that $W\cap C \leq Q_1$.
	Let $x$ and $y$ denote the first and last vertices, respectively, of $Q_1$ that are in $V(W)$ --- and note that it is possible $x=s$ and / or $y=t$.
	For a diagram of this setup, see Figure \ref{fig:main_theorem_setup}.
	Let $Q_1'$ be the path of $G^\nabla$ formed from $Q_1$ by replacing the segment $Q[x,y]$ with the path $x,w^\nabla,y$.
	Then set the cycle $D = Q_1'\cup Q_2$.

	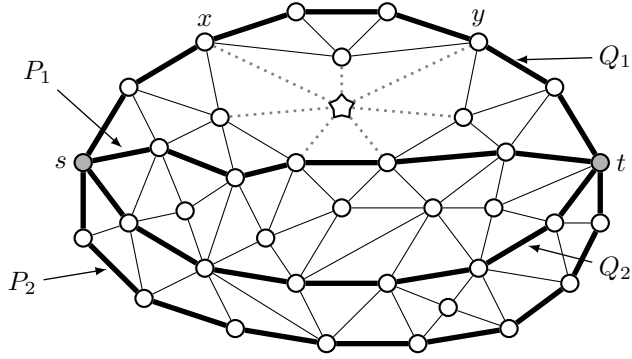
\begin{figure}[h]
		\begin{center}
			\input{main_theorem_setup}
			\caption{An example of four $s-t$ paths, with a large face $W$ crossing $Q_1$. The vertex $w^\nabla$ of $G^\nabla-G$ inside $W$ is star-shaped.}
		\end{center}
		\label{fig:main_theorem_setup}
	\end{figure}
	
	It's clear that $D\cap G \leq C\cap G$ (since $w^\nabla$ is not in $G$), and that $s$ and $t$ are vertices of $D$, proving (1) and (4).
	Note that $W$ covers $D$, but crosses $C$ --- and that no large face can cross $D$ but not $C$.
	Thus $\zeta(D) < \zeta(C)$.
	Similarly, if any large face splits $s$ and $t$ in $D$, then by (1), it splits $s$ and $t$ in $C$.
	Hence no large face splits $s$ and $t$ in $D$, proving (2).
	Because $W$ does not split $s$ and $t$, the vertex $w^\nabla$ is not adjacent to any vertex of $Q_2-\{s,t\}$.
	By the choice of $x$ and $y$, the path $Q_1'$ does not have a chord. 
	By the prior two observations, the cycle $D$ is chordless.
	Finally, let $P$ be a path satisying the hypotheses of (5). 
	Since $D\cap G \leq C\cap G$, we have $P\cap D \leq P\cap C$, so $P$ and $D$ are internally disjoint.
	
	It remains to prove (5).
	Suppose paths $P_1$ and $P_2$ satisfy the hypotheses of (5), and assume without loss of generality that $P_1 \leq \int[C]$ and $P_2\leq \ext[C]$.
	Since $D\cap G \leq C\cap G$, we have $P_i\cap D \leq P_i\cap C$, so $P_i$ and $D$ are internally disjoint (for $i=1,2$).
	As $s$ and $t$ are not adjacent, there exist vertices $u$ of $P_1-\{s,t\}$, $v$ of $P_2-\{s,t\}$, and $z$ in $Q_2 - \{s,t\}$.
	By Propositions \ref{prop:min_sep_3conn_2_components} and \ref{prop:chordless_cycle_is_min_sep}, the cycle $C$ is a minimal cut of $G^\nabla$, and the subgraphs $\int(C)$ and $\ext(C)$ of $G^\nabla$ are its two connected components.
	Similarly, the cycle $C' = P_1\cup P_2$ is a minimal cut.
	Without loss of generality, $Q_2-\{s,t\}$ lies in $\int(C')$ and $Q_1-\{s,t\}$ lies in $\ext(C')$.
	As $P_1$ and $P_2$ are paths in $G$, and $W\cap C\leq Q_1$, we also have that $W\leq \ext[C']$.
	As $C'$ is a minimal cut of $G^\nabla$, the vertex $u$ has a neighbor in both regions of $C'$, and  $\int(C')$ is connected.
	Thus there is a $u-z$ path $R_1$ such that $R_1-\{u\} \leq \int(C')$.
	Similarly, there is a $z-v$ path $R_2$ with $R_2-\{v\} \leq \int(C')$.
	Assume now for the sake of contradiction that $P_1$ and $P_2$ lie in the same region of $D$ (we'll say, without loss of generality, that they lie in the interior).
	Since $D$ is a minimal cut of $G^\nabla$, the interior of $D$ is connected.
	As $C'$ is a minimal cut, the vertices $u$ and $v$ have neighbors in both $\int(C')$ and $\ext(C')$.
	Thus, there is a $u-v$ path $R_3$ in $\int(D)$ that is internally disjoint from $P_1$ and $P_2$.
	Note that by contracting the paths $Q_1', Q_2, R_1, R_2$ and $R_3$, we obtain a $K_5$ minor of $G^\nabla$ with vertices $s, t, u, v$ and $z$, contradicting planarity.
\end{proof}

We can now present our main result.

\begin{thm}
	A 4-connected planar graph is cleavable if and only if it has at most one large face (i.e., if it is a near triangulation).
	\label{thm:4_conn_cleavable_iff_near_mpg}
\end{thm}

\begin{proof}
	Theorem \ref{thm:min_sep_of_near_triangulation} shows that if $G$ has at most one large face, then it is cleavable.
	Assume $G$ has at least two large faces, and let $G^\nabla$ be the auxiliary graph of $G$.
	We will find a disconnected minimal cut of $G$.
	By Corollary \ref{cor:wells_meet_disconn_sep}, we may assume that no two large faces intersect.
	Let $W$ and $X$ be two large faces of $G$ with vertices $w^\nabla$ inside $W$ and $x^\nabla$ inside $X$.
	Note that $|V(W)|, |V(X)| \geq 4$.
	By Menger's Theorem, there exist four disjoint $W-X$ paths $P_1, \dots, P_4$ in $G$.
	Let $P_i$ start at vertex $w_i$ in $W$ and end at $x_i$ in $X$.
	We may assume that each of these paths is chordless, and that $w_1, w_2, w_3, w_4$ appear in that cyclic order on $W$.
	
	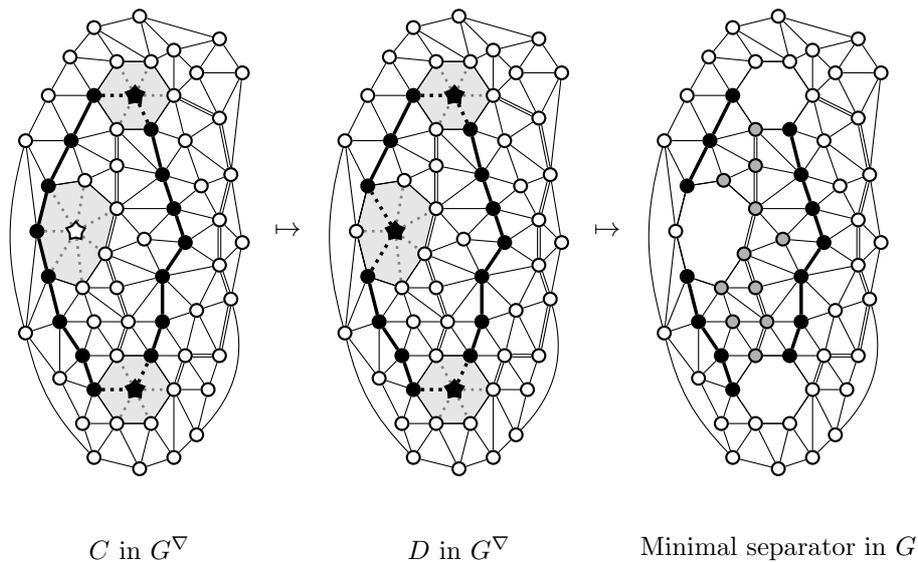
\begin{figure}[h]
		\begin{center}
			\input{algo_steps}
			\caption{A demonstration of how the proof of Theorem \ref{thm:4_conn_cleavable_iff_near_mpg} extracts a disconnected minimal cut from four $W-X$ Menger paths.}
		\end{center}
		\label{fig:algo_steps}
	\end{figure}
	
	In $G^\nabla$, the cycle $C = (w_1, w^\nabla, w_3)\cup P_1 \cup (x_1, x^\nabla, x_3)\cup P_3$ is chordless, and separates $P_2$ from $P_4$.
	Further, due to the two paths $P_2$ and $P_4$ of $G$, no large face of $G$ splits $w^\nabla$ and $x^\nabla$ in $C$.
	Starting with $C$, and repeatedly applying Theorem \ref{thm:crossing_removal_algorithm}, we obtain a cycle $D$ of $G^\nabla$ such that:
	\begin{itemize}
		\item $D$ is chordless, and
		\item $w^\nabla$ and $x^\nabla$ are vertices of $D$, and
		\item no large face of $G$ crosses $D$, and
		\item $D\cap G \leq C\cap G$, and
		\item $D$ Jordan-separates $P_2$ and $P_4$.
	\end{itemize}
	By Theorem \ref{thm:auxillary_min_sep_restricts}, the set $S = V(D)\cap V(G)$ is a minimal cut of $G$, and $S\subseteq V(P_1)\cup V(P_3)$. 
	The vertices of $G^\nabla-G$ form an independent set in $G^\nabla$, so $S\cap V(P_1)\neq \emptyset$ and $S\cap V(P_3)\neq \emptyset$.
	Finally, the sets $S\cap V(P_1)$ and $S\cap V(P_3)$ are not adjacent in $G$ as they are separated by $P_2 \cup P_4$. 
	Thus $G[S]$ has at least two components, completing the proof.
\end{proof}

By Theorem \ref{thm:min_sep_of_near_triangulation}, every 2-connected graph with at most one large face is cleavable. 
However, Theorem \ref{thm:4_conn_cleavable_iff_near_mpg} does not hold if we weaken the connectivity requirement: the graph in Figure \ref{fig:3conn_cleavable_inter_wells} is 3-connected, cleavable, and has multiple large faces. 
The example in Figure \ref{fig:3conn_cleavable_inter_wells} is somewhat trivial --- many of the large faces are adjacent, and removing the curved edges yields a graph that is still 3-connected, but has only one large face.
As the dedicated reader may check, the graph in Figure \ref{fig:3conn_counter_ex} is 3-connected, cleavable, and contains two non-intersecting large faces. 
Further, this graph is not spanned by any 3-connected planar subgraph with a single large face.

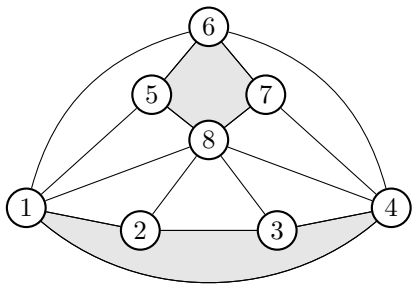
\begin{figure}[h]
	\begin{center}
		\input{3conn_counter_example}
		\caption{A 3-connected, cleavable graph.}
		\label{fig:3conn_counter_ex}
	\end{center}
\end{figure}

\section{The Auxiliary graph and planar supergraphs}
\label{sec:auxillary}

Throughout the proof of the main result, Theorem \ref{thm:4_conn_cleavable_iff_near_mpg}, we make use of the auxiliary graph. 
In this section, we examine this tool further.
The results here (in tandem with Proposition \ref{prop:min_sep_mpg_is_cycle}) help demonstrate that the auxiliary graph is a convenient tool for studying cuts in 3-connected planar graphs.

\begin{lem}
	Let $G$ and $H$ be 3-connected plane graphs such that $G$ is an induced subgraph of $H$.
	If $R$ is a minimal cut of $G$, then there exists a minimal cut $S$ of $H$ such that $R\subseteq S$.
	\label{lem:min_separator_extends}
\end{lem}

\begin{proof}
	By Propositions  \ref{prop:min_sep_is_adjacent} and \ref{prop:min_sep_3conn_2_components}, $G-R$ has two components $C_1$ and $C_2$, and every vertex of $R$ has a neighbor in each of these components.
	Clearly every $C_1-C_2$ separator in $H$ must contain $R$.
	Further, since $G$ is an induced subgraph of $H$, $R\cup (V(H)-V(G))$ is a $C_1-C_2$ separator in $H$.
	Thus there exists a minimal $C_1-C_2$ separator $S$ in $H$ such that $R\subseteq S \subseteq R\cup (V(H)-V(G))$.
	By Proposition \ref{prop:min_ABsep_is_min_sep}, $S$ is a minimal cut of $H$.
\end{proof}

If $H= G^\nabla$, the cut $S$ promised by Lemma \ref{lem:min_separator_extends} is unique.
As the proof shows, we can easily see which vertices of $G^\nabla-G$ belong to $S$.

\begin{thm}
	Suppose $G$ is a 3-connected plane graph.
	For every minimal cut $R$ of $G$, there exists a {\em unique} minimal cut $S$ of $G^\nabla$ such that $S\cap V(G) = R$.
	\label{thm:min_sep_extends_uniquely_augmented}
\end{thm}

\begin{proof}
	Suppose $R$ is a minimal cut of $G$. 
	By Lemma \ref{lem:min_separator_extends}, $R$ extends to a minimal cut $S$ of $G^\nabla$ --- and $S$ induces a chordless cycle $C = G^\nabla[S]$ per Proposition \ref{prop:min_sep_mpg_is_cycle}.
	Let $W$ be (the cycle bounding) an arbitrary large face of $G$, and let $v_W$ be the vertex of $G^\nabla-G$ in the large face. 
	To prove that $S$ is the unique minimal cut of $G^\nabla$ extending $R$, it suffices to prove the following claim:
	
	\textit{Claim:} The vertex $v_W$ is in $S$ if and only if $W\cap C$ consists of two non-adjacent vertices.\\
	Suppose $v_W \in S$, and let $C = u_1, \dots, u_i, v_W = u_{i+1}, u_{i+2}, \dots, u_k$.
	Since $C$ is an induced cycle, $W\cap C$ consists of only the vertices $u_i$ and $u_{i+2}$ --- and these vertices are not adjacent.
	Conversely, suppose $W\cap C = \{u,w\}$, where $u$ and $w$ are not adjacent.
	Assume to the contrary that $v_W$ is not in $S$.
	Then $u$ and $w$ partition $C$ into two internally disjoint $u-w$ paths $C_1$ and $C_2$. 
	Similarly, $u$ and $w$ partition $W$ into internally disjoint paths $W_1$ and $W_2$. 
	Let $x_1$ and $x_2$ be vertices of $\ins{W_1}$ and $\ins{W_2}$, respectively.
	The Jordan curve $J$ formed from the path $C_1$ by adding a $u-w$ arc through $W$ separates $x_1$ and $x_2$. 
	Further, $J\cap G$ is a proper subgraph of $G[R]$, contradicting the minimality of $R$ and proving the claim.
\end{proof}

\begin{rem}
	The auxiliary graph $G^\nabla$ is the smallest maximal planar graph containing $G$ as an induced subgraph --- in the sense that every maximal planar graph $H$ containing $G$ as an induced subgraph has a $G^\nabla$ minor.
\end{rem}

To see the above remark, contract all the components of $H-G$ down to a single vertex to obtain $G^\nabla$.

\section{Components of minimal cuts}
\label{sec:components}

In this section, we find constraints on the number and type of components of a disconnected minimal cut in a 4-connected planar graph.

\begin{thm}
	Let $G$ be a 2-connected plane graph with $k$ large faces.
	If $R$ is a minimal cut of $G$, then $G[R]$ has at most $k$ components.
	\label{thm:num_components_num_wells}
\end{thm}

\begin{proof}
	Let $R$ be a minimal cut of $G$, and $G^\nabla$ the auxiliary graph of $G$.
	By Lemma \ref{lem:min_separator_extends}, there is some minimal cut $S\subseteq V(H)$ of $G^\nabla$ such that $R\subseteq S$.
	Since $G^\nabla$ is a maximal planar graph, $H[S]$ is a chordless cycle by Proposition \ref{prop:chordless_cycle_is_min_sep}.
	Each vertex of $G^\nabla-G$ is the single vertex inside of a large face, and therefore the cycle $S-R$ contains at most $k$ vertices.
	Therefore $G[R] = H[R]$ is a cycle from which at most $k$ vertices have been removed, and thus $G[R]$ has at most $k$ components.
\end{proof}

We briefly comment on the problem of determining whether or not a 3-connected planar graph has a stable minimal cut.

\begin{prop}
	If $G$ is a 3-connected plane graph such that no two large faces of $G$ intersect, then $G$ does not have a stable minimal cut.
	\label{prop:no_stable_min_sep}
\end{prop}

\begin{proof}
	Let $R$ be a minimal cut of $G$.
	Combining Proposition \ref{prop:min_sep_mpg_is_cycle} and Lemma \ref{lem:min_separator_extends}, we see that $R$ is contained in a cycle $C$ of the auxiliary graph $G^\nabla$.
	The induced graph $G[R]$ is either the single cycle $C$, or a union of paths divided by vertices of $G^\nabla-G$. 
	Since no two large faces are adjacent, each of these paths has at least one edge.
	Thus $R$ is not an independent set.
\end{proof}

Applying Proposition \ref{prop:wells_dont_intersect}, we can say more for 4-connected plane graphs.

\begin{cor}
	If $G$ is a 4-connected plane graph such that $\n(u)$ is connected for each vertex $u$ of $G$, then $G$ does not have a stable minimal cut.
	\label{cor:no_stable_sep_when_nhoods_conn}
\end{cor}

\section{\textsc{Minimal Disconnected Cut} witness in $\bm{\mathcal{O}(n)}$ time}
\label{sec:algorithm}

Throughout this section, $n$ refers to the vertex count of the planar graph in question, and $m$ its number of edges.
In \cite{min_disc_cuts_in_pgs}, it is shown that \textsc{Minimal Disconnected Cut} can be decided in $\mathcal{O}(n^3)$ time for 3-connected planar graphs.
In this section, we present an $\mathcal{O}(n)$ algorithm \textsc{MinDiscCut} for finding a minimal disconnected cut of a 4-connected planar graph (or determining that one does not exist). 
We assume that the reader is familiar with some of the basic data structures for representing a plane graph, such as the \textbf{combinatorial embedding} and \textbf{doubly connected edge list} (DCEL) --- see \cite{comp_geometry_preparata, finding_polyhedra_intersection}.
Unless stated otherwise, we always assume graphs are represented as adjacency lists with the vertex set represented as an array $V = [1,2, \dots, n]$, and that paths and large faces are represented as lists of vertices.
We denote by $\textup{\textbf{Adj}}\bm{(i)}$ the list of vertices adjacent to vertex $i$.

The 1974 paper \cite{hopcroft_tarjan_planarity_test} presents a linear-time algorithm (henceforth the \textbf{Hopcroft-Tarjan algorithm}) for deciding whether a graph (input as an adjacency list) is planar.
The algorithm is not simple --- \cite{hopcroft_tarjan_kocay_exposition} provides further exposition.
An extension of the Hopcroft-Tarjan algorithm by Mehlhorn and Mutzel not only decides whether the input graph is planar, but also provides a planar combinatorial embedding of the graph (in linear time) if it is planar \cite{hopcroft_tarjan_embedding_phase}.
In \cite{hopcroft_tarjan_embedding_implementation}, an implementation of the extended algorithm is outlined.
The Hopcroft-Tarjan algorithm is not the final word on planarity testing and plane embeddings --- a number of other $\mathcal{O}(n)$ time algorithms for embedding a plane graph have been developed.
See for example \cite{pq_tree_planar_embedding}, \cite{simplified_o_n_planarity} and the literature discussion of \cite{hopcroft_tarjan_embedding_phase}.

\begin{thm}[\textsc{Embed}]\textup{\cite{hopcroft_tarjan_planarity_test, hopcroft_tarjan_embedding_phase, pq_tree_planar_embedding, simplified_o_n_planarity}.}
	There exists an $\mathcal{O}(n)$ algorithm \textsc{Embed} that takes as input the adjacency list of a planar graph, and returns a combinatorial embedding of the graph.
	\label{thm:plane_embedding_linear_time}
\end{thm}

Given a combinatorial embedding of a 2-connected planar graph, we can construct a DCEL for the graph in linear time \cite{comp_geometry_preparata} (for an algorithm to do so, see \cite{finding_polyhedra_intersection}).
With a DCEL, it is easy to scan each face (in time $\mathcal{O}(|V(f)|)$ per face $f$) and create a list containing the vertex sets of each face in time $\mathcal{O}(\sum |V(f)|) = \mathcal{O}(n)$.

\begin{lem}[\textsc{ListLargeFaces}]
	There is an $\mathcal{O}(n)$ algorithm \textsc{ListLargeFaces} whose input is a combinatorial embedding of a 2-connected planar graph $G$, that outputs a list $\Omega = (W_1, \dots, W_k)$ of all the large faces of $G$ --- where each $W_i$ is a list of the large face's vertices as they appear in clockwise order around the face.
	\label{lem:listwells_linear_time}
\end{lem}

\begin{cor}
	The decision problem \textsc{Minimal Disconnected Cut} can be solved in linear time for 4-connected planar graphs.
	\label{cor:min_disc_cut_decide_linear_time}
\end{cor}

\begin{proof}
	To decide whether a 4-connected planar graph has a disconnected minimal cut, we can find a combinatorial embedding using \textsc{Embed}, and then count the large faces using \textsc{ListLargeFaces}. 
	Thus, the result follows by Theorems \ref{thm:4_conn_cleavable_iff_near_mpg} and \ref{thm:plane_embedding_linear_time}, and Lemma \ref{lem:listwells_linear_time}.
\end{proof}

Let $s,t$ be two non-adjacent vertices of a graph, and $\kappa$ the minimum cardinality of a $s-t$ separator. 
By applying the Ford-Fulkerson flow algorithm to a modified version of the graph, a collection of $\kappa$ internally disjoint $s-t$ paths can be found in $\mathcal{O}(m\kappa)$ time \cite{ford_fulkerson} (for a modern discussion and implementation, see \cite{intro_to_algos}).
In planar graphs this time can be improved --- there exists an $\mathcal{O}(n)$ algorithm for finding a maximum collection of intenally disjoint $s-t$ paths \cite{menger_paths_in_pg} (we attain the same big-Oh time when $\kappa$ is a fixed constant using the flow algorithm).

\begin{thm}[\textsc{MengerPaths}]\textup{\cite{menger_paths_in_pg}}
	There exists an $\mathcal{O}(n)$ algorithm \textsc{MengerPaths} that takes a combinatorial embedding of a planar graph, and two vertices $s$ and $t$, and returns a maximum cardinality collection of internally disjoint $s-t$ paths.
	\label{thm:menger_linear_time}
\end{thm}

We need some basic tools and lemmas to show parts of \textsc{MinDiscCut} run in linear time. 
Let $V = [1,2,\dots, n]$ be a fixed ground set, and $A = (a_1, \dots, a_k)$ a list of elements of $V$ with no repeated entries.
Denote by $\bm{\textup{\textbf{Subset}}(A)}$ the $n$-entry array whose $i^{th}$ entry is 1 if $i$ appears in the list $A$, and $0$ if not.
Similarly, let $\textup{\textbf{Lookup}}\bm{(A)}$ be the $n$-entry array whose $i^{th}$ entry is the integer $j$ if $a_j = i$, and is $0$ if $i$ does not appear in $A$.
For example, if $V=[1,\dots, 8]$ and $A = (3,1,8,5)$, then $\textup{Subset}(A) = [1,0,1,0,1,0,0,1]$ and $\textup{Lookup}(A) = [2,0,1,0,4,0,0,3]$. 
It's clear that both arrays can be initialised in $\mathcal{O}(n)$ time.
Given a list $A$ and an element $i$ of $V$, we denote by $\textup{\textbf{Append}}\bm{(A, i)}$ the action of appending the single element $i$ to the end of the list $A$.

\begin{lem}[\textsc{TruncatePath}]
	Consider a (directed) graph $G$, two disjoint sets $A$ and $B$ of vertices, and a path $P$ starting in $A$ and ending in $B$.
	There is an $\mathcal{O}(n)$ algorithm \textsc{TruncatePath} that finds a subpath $Q = q_1, \dots, q_k$ of $P$ that starts in $A$, ends in $B$, and is internally disjoint from $A\cup B$. 
	\label{lem:truncate_paths_linear_time}
\end{lem}

\begin{proof}
	We outline the \textsc{TruncatePath} algorithm that takes the path $P = (p_1, \dots, p_k)$ and disjoint sets $A$ and $B$ of vertices, and returns the truncated path $Q$ (see Algorithm \ref{alg:truncate_path}).
	
	\begin{algorithm}[h]
		\caption{\textsc{TruncatePath$(A,B,P)$}}\label{alg:truncate_path}
		\SetKw{EndLoop}{End loop}
		
		\Indm
		\BlankLine
		
		\Indp
		$X\gets \Subset(A)$, $Y\gets \Subset(B)$, $v_a\gets p_1$, $v_b\gets p_k$\\
		\For{$p_i$ $\In$ $P$}{
			\If{$X[p_i]=1$}{$v_a\gets p_i$}
			\If{$Y[p_i]=1$}{$v_b\gets p_i$\\
				\EndLoop}}
		$Q\gets P[v_a, v_b]$, $\Return$ $Q$\\
	\end{algorithm}
	We follow the path $P$, tracking the highest index vertex $v_a$ of $A$ in $P$, until we encounter a vertex $v_b$ of $B$. Then return the path segment $P[v_a, v_b]$.
	It's clear this runs in $\mathcal{O}(n)$ time.
\end{proof}

\begin{lem}[\textsc{RemoveChords}]
	Let $G$ be a connected (directed) graph, and $P = p_1, \dots, p_k$ a path in $G$.
	There is an $\mathcal{O}(n+m)$ algorithm \textsc{RemoveChords} that takes $G$ and $P$ as inputs, and returns a chordless $p_1-p_k$ path $Q$ such that $Q\leq P$.
	\label{lem:remove_chords_linear_time}
\end{lem}

\begin{proof}
	We outline the \textsc{RemoveChords} algorithm.
	It's input is a (directed) graph $G$, and a path $P = (p_1, \dots, p_k)$.
	The output is a new chordless path $Q$ (see Algorithm \ref{alg:remove_chords}).
	
	\begin{algorithm}[h]
		\caption{\textsc{RemoveChords$(G,P)$}}\label{alg:remove_chords}
		
		\Indm
		\BlankLine
		
		\Indp
		$X\gets \Lookup(P)$, $Q\gets (p_1)$, $i\gets 1$\\
		\While{$i<k$}{
			$j\gets 0$\\
			\For{$v$ \In $\adj(p_i)$}{
				\If{$X[v] > j$}{$j\gets X[v]$\\}}
			$\Append(Q,p_j)$\\
			$i\gets j$\\}
		$\Return$ $Q$\\
	\end{algorithm}
	The algorithm creates a new path $Q = p_1$. 
	It finds the neighbour $p_j$ of $p_1$ that is furthest along the path $P$, and appends $p_j$ to the end of $Q$ --- then keeps repeating this step with $p_j$ until it eventually hits the last vertex of $P$.
	The initialization takes time $\mathcal{O}(n)$.
	The while loop, at worst, scans over all the adjacency lists $\adj(i)$ once, for a total running time of $\mathcal{O}(\sum_{i\in V} |\adj(i)|) = \mathcal{O}(m)$.
\end{proof}

\begin{lem}[\textsc{FaceIntersection}]
	Let $G=(V,E)$ be a 2-connected graph, and $\Omega = (W_1, \dots, W_k)$ a list of all large faces of $G$, where each $W_i$ is a list of the large face's vertices.
	There exists an $\mathcal{O}(n)$ algorthm \textsc{FaceIntersection} with inputs $V$ and $\Omega$, that returns $\texttt{null}$ if no two large faces intersect, and returns some vertex $v$ in the intersection of two large faces otherwise.
	\label{lem:well_intersection_linear_time}
\end{lem}

\begin{proof}
	We outline the \textsc{FaceIntersection} algorithm (Algorithm \ref{alg:well_intersection}).
	The input is a vertex set $V = [1,\dots, n]$ and a list $\Omega = (W_1,\dots, W_k)$ of large faces represented as vertex lists.
	The output is a vertex $v$ in the intersection of two large faces if such exists, else \texttt{null}.
	The algorithm creates an empty set $X$, and scans over all the large faces in $\Omega$ --- adding each vertex encountered to $X$. 
	If it ever encounters the same vertex $v$ twice, it immediately returns that vertex. 
	
	\begin{algorithm}[h]
		\caption{\textsc{FaceIntersection$(V, \Omega)$}}\label{alg:well_intersection}
		
		\Indm
		\BlankLine
		
		\Indp
		$X\gets \Subset(\emptyset)$, $meet \gets \texttt{null}$\\
		\For{$W_i$ $\In$ $\Omega$}{
			\For{$v$ $\In$ $W_i$}{
				\If{$X[v]=1$}{
					$meet \gets v$\\
					\Return meet, \End\\}
				\Else{$X[v] \gets 1$\\}}}
		\Return $meet$\\
	\end{algorithm}
	The total running time of the for loop in line 4 is $\mathcal{O}(\sum_{W_i\in \Omega} |W_i|)$. 
	In a 2-connected plane graph, the total number of vertex encounters when scanning over all faces of the graph once is at most $2|E|$. 
	So $\mathcal{O}(\sum_{W_i\in \Omega} |W_i|) = \mathcal{O}(m)$, and $\mathcal{O}(m) = \mathcal{O}(n)$ for planar graphs.
\end{proof}

Let $G = (V,E)$ be a graph, $P = p_1, \dots, p_j$ a finite sequence of vertices of $G$ (such as a path), and $\Omega = (W_1,\dots, W_k)$ a list of pairwise disjoint subsets of $V$.
A vertex $p_y$ is a \textbf{middle vertex} if there exist integers $i, x$ and $z$ such that $x<y<z$, and $p_x, p_y, p_z$ are all vertices of $P\cap W_i$.
A \textbf{skipper} of $P$ relative to $\Omega$ is a maximal subsequence of $P$ that does not contain any middle vertices.
In Theorems \ref{thm:crossing_removal_algorithm} and \ref{thm:4_conn_cleavable_iff_near_mpg}, we create skippers $Q_1'$ and $Q_2'$ of the paths $Q_1$ and $Q_2$, relative to the set of large faces. 

\begin{lem}[\textsc{PathSkipper}]
	Let $G = (V,E)$ be a connected graph, $P$ a path in $G$, and $\Omega = (W_1,\dots, W_k)$ a list of pairwise disjoint subsets of $V$.
	There is an $\mathcal{O}(n+m)$ algorithm \textsc{PathSkipper} that takes these inputs, and returns a skipper of $P$ relative to $\Omega$.
	\label{lem:path_skipper_linear_time}
\end{lem}

\begin{proof}
	The algorithm \textsc{PathSkipper} (Algorithm \ref{alg:path_skipper}) takes as input a graph $G$, a list $\Omega = (W_1,\dots, W_k)$ of disjoint subsets represented as vertex lists, and a path $P=(p_1, \dots, p_j)$.
	It returns a skipper $Q$ of the path $P$ relative to $\Omega$.
	It first extends $G$ to a new graph $H$ --- by adding a directed arc from the first vertex of $P\cap W_i$ to the last vertex of $P\cap W_i$ for each large face $W_i$.
	Then, we run the \textsc{RemoveChords} algorithm of Lemma \ref{lem:remove_chords_linear_time} on the extended graph $H$.
	
	\begin{algorithm}[h]
		\caption{\textsc{PathSkipper$(G,P,\Omega)$}}\label{alg:path_skipper}
		
		\Indm
		\BlankLine
		
		\Indp
		$X\gets \Lookup(P)$, $H\gets G$\\
		\For{$W_i$ $\In$ $\Omega$}{
			$a\gets \infty$, $b\gets 0$\\
			\For{$v$ $\In$ $W_i$}{
				\If{$0 < X[v] < a$}{
					$a\gets X[v]$\\}
				\If{$X[v] > b$}{
					$b\gets X[v]$\\}}
			\If{$b\neq 0$}{
				$\Append(\adj_H(p_a), p_b)$\\}}
		$Q\gets \RemoveChords(H, P)$, $\Return$ $Q$\\
	\end{algorithm}

	Initialising the algorithm takes time $\mathcal{O}(n+m)$.
	Since the subsets $W_i$ are disjoint, the algorithm visits each vertex $v$ at most once, and $H$ has less than $n+m$ edges.
	Thus the call to \textsc{RemoveChords} runs in time $\mathcal{O}(n+m)$ per Lemma \ref{lem:remove_chords_linear_time}.
\end{proof}

We now outline the linear time algorithm, \textsc{MinDiscCut}, for finding a minimal disconnected cut of a 4-connected planar graph. 

\begin{thm}
	There is an $\mathcal{O}(n)$ algorithm \textsc{MinDiscCut} that finds a minimal disconnected cut of a 4-connected planar graph if one exists --- and returns \texttt{null} if not.
	\label{thm:min_disc_cut_linear_time}
\end{thm}

\begin{proof}
	the algorithm \textsc{MinDiscCut} below takes a 4-connected planar graph $G=(V,E)$, and returns a minimal disconnected cut $S$ of $G$ if one exists, \texttt{null} if not.
	We break the algorithm into two parts.
	
	\begin{algorithm}[h]
		\caption{\textsc{MinDiscCut$(G)$} part 1}\label{alg:min_disc_cut}
		
		\Indm
		\BlankLine
		
		\Indp
		$S\gets \texttt{null}$\\
		\BlankLine
		
		Run \Embed to get a combinatorial embedding of $G$, and \ListLargeFaces to return a list $\Omega = (W_1,\dots, W_k)$ of large faces of the embedding.\\
		\If{$|\Omega| < 2$}{\Return \texttt{null}, \End\\}
		\BlankLine
		
		Run $\FaceIntersection(V,\Omega)$, which returns $meet$.\\
		\If{$meet \neq \texttt{null}$}{
			$S\gets \adj(meet)$\\
			\Return $S$, \End\\}
		\BlankLine
	\end{algorithm}
	It first finds a combinatorial embedding of $G$, and then uses this embedding to create a list $\Omega = (W_1,W_2,\dots)$ of all the large faces of $G$.
	If $G$ has less than two large faces, it returns \texttt{null}, indicating that $G$ is cleavable, per Theorem \ref{thm:4_conn_cleavable_iff_near_mpg}.
	\textsc{MinDiscCut} then checks the pairwise intersections of the large faces. 
	If there exists a vertex $v$ in two large faces, the algorithm returns $\n(v)$, which is a disconnected minimal cut per Corollary \ref{cor:well_meet_is_min_disc_cut}.
	We let $x_1,\dots, x_4$ denote the first four vertices of $W_1$, and $y_1,\dots, y_4$ the first four vertices of $W_2$.
	
	\begin{algorithm}[h]
		\caption{\textsc{MinDiscCut$(G)$} part 2}
		\setcounter{AlgoLine}{9}
		\Indm
		\Indp
		Create a new graph $G'$ by adding vertices $s$ and $t$ to $G$ --- where $s$ is adjacent to $x_1,\dots, x_4$ and $t$ is adjacent to $y_1, \dots, y_4$.\\
		
		Run $\MengerPaths(G',s,t)$, returning four internally disjoint $s-t$ paths $P_1', \dots, P_4'$.\\
		
		Remove $s$ and $t$ from each $P_i$ to get four $W_1-W_2$ paths $P_1,\dots, P_4$ of $G$.\\
		
		On each $P_i$, run $\TruncatePath(P_i, W_1, W_2)$, and then \RemoveChords, returning chordless paths $R_i$ that are internally disjoint from $W_1\cup W_2$.\\
		
		
		Run $\PathSkipper(G, R_i, \Omega)$ for $R_1$, $R_3$ to obtain skippers $S_1$, $S_3$.\\
		
		$S\gets S_1\cup S_3$, $\Return$ $S$\\
	\end{algorithm}
	In the second part, \textsc{MinDiscCut} finds four internally disjoint, chordless paths $R_1,\dots, R_4$ that start in $W_1$, end in $W_2$, and are internally disjoint from both.
	Finally, the algorithm finds skippers for $R_1$ and $R_3$ relative to $\Omega$. 
	The union of these two skippers is the restriction of a chordless cycle of $G^\nabla$ --- and is thus the desired minimal disconnected cut per Theorem \ref{thm:auxillary_min_sep_restricts} (also see the proof of Theorem \ref{thm:crossing_removal_algorithm}).
	
	In a planar graph, $m \leq 3n-6$, so all the subroutine calls (\textsc{Embed}, \textsc{RemoveChords}, etc.) run in $\mathcal{O}(n)$ time by the prior results of this section.
	Further, its clear that all the additional steps (such as the constructions in lines 10 and 12) can be done in linear time.
\end{proof}

\section{Conclusion and further questions}

We have shown that the \MDC{} problem is easily solved for 4-connected planar graphs: the graph has a minimal disconnected cut if and only if it has at least two large faces. 
This characterization, combined with other strong properties of 4-connected planar graphs, allows us to solve the \MDC{} problem --- and produce a witnessing cut --- in linear time.

This investigation into problem raises a handful of further questions:
\begin{itemize}
	\item Can the $\mathcal{O}(n^3)$ time for solving \MDC{} on 3-connected planar graphs in \cite{min_disc_cuts_in_pgs} can be improved?
	
	\item Is there some minimum connectivity $\kappa(g)$ such that the \MDC{} problem can be easily solved for graphs of (orientable) genus $g$ that are $\kappa(g)$-connected? When $g=0$, we have shown that $\kappa = 4$ is sufficient.
	
	\item Planar graphs that are 4-connected do not have a particularly simple structure. It therefore seems possible that there are many graph classes of interest for which it is easy to characterise the cleavable graphs, and / or solve the \MDC{} and \DC{} problems in polynomial time. For what other graph classes can this be done?
\end{itemize}
 
\section{Acknowledgments}
Many thanks to Guy Paterson-Jones for valuable discussion --- and to David Erwin for general advice about the paper.

\end{document}

%% file: jordan_separating.tex
\begin{tikzpicture}
	[scale=0.5,inner sep=0.7mm, 
	blackvertex/.style={circle,thick,draw, fill=black},
	whitevertex/.style={circle,thick,draw, fill=white},
	none/.style={},
	red/.style={line width=2pt}] 
	
		\node [style=blackvertex] (0) at (0, 2.75) {};
		\node [style=blackvertex] (1) at (-1.75, 1.5) {};
		\node [style=blackvertex] (2) at (-1.75, -0.25) {};
		\node [style=blackvertex] (3) at (0, -1.5) {};
		\node [style=blackvertex] (4) at (1.75, -0.25) {};
		\node [style=blackvertex] (5) at (1.75, 1.5) {};
		\node [style=whitevertex] (6) at (-1, 0.75) {};
		\node [style=whitevertex] (7) at (0.75, 0.75) {};
		\node [style=whitevertex] (8) at (0, -0.25) {};
		\node [style=whitevertex] (9) at (0, 1.5) {};
		\node [style=whitevertex] (10) at (-1.75, -1.75) {};
		\node [style=whitevertex] (11) at (-3.25, -0.75) {};
		\node [style=whitevertex] (12) at (-3, 0.75) {};
		\node [style=blackvertex] (13) at (5.25, 2.75) {};
		\node [style=blackvertex] (14) at (3.5, 1.5) {};
		\node [style=blackvertex] (15) at (3.5, -0.25) {};
		\node [style=blackvertex] (16) at (5.25, -1.5) {};
		\node [style=blackvertex] (17) at (7, -0.25) {};
		\node [style=blackvertex] (18) at (7, 1.5) {};
		\node [style=whitevertex] (19) at (4.25, 0.75) {};
		\node [style=whitevertex] (20) at (6, 0.75) {};
		\node [style=whitevertex] (21) at (5.25, -0.25) {};
		\node [style=whitevertex] (22) at (5.25, 1.5) {};
		\node [style=whitevertex] (29) at (8.75, 0.75) {};
		\node [style=whitevertex] (30) at (10.5, 0.75) {};
		\node [style=whitevertex] (31) at (9.75, -0.25) {};
		\node [style=whitevertex] (32) at (9.75, 1.5) {};
		\node [style=none] (33) at (0, -3) {$G$};
		\node [style=none] (34) at (5.25, -3) {$\int[C]$};
		\node [style=none] (35) at (9.75, -3) {$\int(C)$};

		\draw [style=red] (0) to (1);
		\draw [style=red] (1) to (2);
		\draw [style=red] (2) to (3);
		\draw [style=red] (3) to (4);
		\draw [style=red] (5) to (4);
		\draw [style=red] (0) to (5);
		\draw (9) to (6);
		\draw (6) to (8);
		\draw (8) to (7);
		\draw (9) to (7);
		\draw (0) to (9);
		\draw (1) to (6);
		\draw (6) to (2);
		\draw (6) to (7);
		\draw (7) to (4);
		\draw (8) to (4);
		\draw (8) to (3);
		\draw (1) to (12);
		\draw (12) to (2);
		\draw (12) to (11);
		\draw (11) to (2);
		\draw (2) to (10);
		\draw (11) to (10);
		\draw (10) to (3);
		\draw [style=red] (13) to (14);
		\draw [style=red] (14) to (15);
		\draw [style=red] (15) to (16);
		\draw [style=red] (16) to (17);
		\draw [style=red] (18) to (17);
		\draw [style=red] (13) to (18);
		\draw (22) to (19);
		\draw (19) to (21);
		\draw (21) to (20);
		\draw (22) to (20);
		\draw (13) to (22);
		\draw (14) to (19);
		\draw (19) to (15);
		\draw (19) to (20);
		\draw (20) to (17);
		\draw (21) to (17);
		\draw (21) to (16);
		\draw (32) to (29);
		\draw (29) to (31);
		\draw (31) to (30);
		\draw (32) to (30);
		\draw (29) to (30);

\end{tikzpicture}

%% file: near_triangulation.tex
\begin{tikzpicture}
	[scale=0.5,inner sep=0.7mm, 
	blackvertex/.style={circle,thick,draw, fill=black},
	whitevertex/.style={circle,thick,draw, fill=white},
	none/.style={},
	red/.style={line width=2pt},
	lightedge/.style={line width=1pt, color=black!50, dotted}]
	
		\node [style=whitevertex] (0) at (0, 5) {};
		\node [style=whitevertex] (1) at (-4, -1) {};
		\node [style=whitevertex] (2) at (4, -1) {};
		\node [style=blackvertex] (3) at (-2.5, 2) {};
		\node [style=blackvertex] (4) at (-2.75, 0.25) {};
		\node [style=blackvertex] (5) at (-2, -0.75) {};
		\node [style=blackvertex] (6) at (-0.5, -0.25) {};
		\node [style=blackvertex] (7) at (-0.75, 1.75) {};
		\node [style=whitevertex] (8) at (-0.75, 3) {};
		\node [style=whitevertex] (9) at (2, 2.5) {};
		\node [style=whitevertex] (10) at (0.75, 1) {};
		\node [style=whitevertex] (11) at (0.75, -0.5) {};
		\node [style=whitevertex] (12) at (0.5, 2.5) {};
		\node [style=whitevertex] (13) at (2, 0.25) {};
		\node [style=whitevertex] (14) at (-0.5, -1.75) {};
		\node [style=none] (15) at (-1.75, 0.75) {$f_W$};
		\node [style=whitevertex] (16) at (9.75, 5) {};
		\node [style=whitevertex] (17) at (5.75, -1) {};
		\node [style=whitevertex] (18) at (13.75, -1) {};
		\node [style=whitevertex] (24) at (9, 3) {};
		\node [style=whitevertex] (25) at (11.75, 2.5) {};
		\node [style=whitevertex] (26) at (10.5, 1) {};
		\node [style=whitevertex] (27) at (10.5, -0.5) {};
		\node [style=whitevertex] (28) at (10.25, 2.5) {};
		\node [style=whitevertex] (29) at (11.75, 0.25) {};
		\node [style=whitevertex] (30) at (9.25, -1.75) {};
		\node [style=whitevertex] (31) at (7.25, 2) {};
		\node [style=whitevertex] (32) at (7, 0.25) {};
		\node [style=whitevertex] (33) at (7.75, -0.75) {};
		\node [style=whitevertex] (34) at (9.25, -0.25) {};
		\node [style=whitevertex] (35) at (9, 1.75) {};
		\node [style=whitevertex, shape=star, star points = 5] (36) at (8, 0.75) {};
		\node [style=none] (38) at (0, -3.5) {$G$};
		\node [style=none] (39) at (9.75, -3.5) {$G^\nabla$};

		\draw [bend right] (0) to (1);
		\draw [bend right] (1) to (2);
		\draw [bend left] (0) to (2);
		\draw (1) to (4);
		\draw [style=red] (3) to (4);
		\draw [style=red] (4) to (5);
		\draw [style=red] (5) to (6);
		\draw [style=red] (7) to (6);
		\draw [style=red] (3) to (7);
		\draw (3) to (1);
		\draw (0) to (3);
		\draw (8) to (3);
		\draw (8) to (7);
		\draw (0) to (8);
		\draw (1) to (5);
		\draw (1) to (14);
		\draw (14) to (2);
		\draw (5) to (14);
		\draw (6) to (14);
		\draw (0) to (9);
		\draw (9) to (2);
		\draw (0) to (12);
		\draw (9) to (12);
		\draw (8) to (12);
		\draw (8) to (10);
		\draw (12) to (10);
		\draw (7) to (10);
		\draw (7) to (11);
		\draw (6) to (11);
		\draw (10) to (11);
		\draw (11) to (14);
		\draw (11) to (2);
		\draw (11) to (13);
		\draw (13) to (2);
		\draw (9) to (13);
		\draw (9) to (10);
		\draw (10) to (13);
		\draw [bend right] (16) to (17);
		\draw [bend right] (17) to (18);
		\draw [bend left] (16) to (18);
		\draw (16) to (24);
		\draw (17) to (30);
		\draw (30) to (18);
		\draw (16) to (25);
		\draw (25) to (18);
		\draw (16) to (28);
		\draw (25) to (28);
		\draw (24) to (28);
		\draw (24) to (26);
		\draw (28) to (26);
		\draw (26) to (27);
		\draw (27) to (30);
		\draw (27) to (18);
		\draw (27) to (29);
		\draw (29) to (18);
		\draw (25) to (29);
		\draw (25) to (26);
		\draw (26) to (29);
		\draw (31) to (17);
		\draw (32) to (17);
		\draw (33) to (17);
		\draw (33) to (30);
		\draw (34) to (30);
		\draw (34) to (27);
		\draw (35) to (27);
		\draw (35) to (26);
		\draw (24) to (35);
		\draw (31) to (24);
		\draw (31) to (32);
		\draw (32) to (33);
		\draw (34) to (33);
		\draw (35) to (34);
		\draw (31) to (35);
		\draw (16) to (31);
		\draw [style=lightedge] (31) to (36);
		\draw [style=lightedge] (35) to (36);
		\draw [style=lightedge] (32) to (36);
		\draw [style=lightedge] (36) to (33);
		\draw [style=lightedge] (36) to (34);
\end{tikzpicture}

%% file: min_seps_near_mpg.tex
\begin{tikzpicture}
	[scale=0.6,inner sep=0.8mm, 
	blackvertex/.style={circle,thick,draw, fill=black},
	whitevertex/.style={circle,thick,draw, fill=white},
	none/.style={},
	red/.style={line width=2pt}]
	
		\node [style=blackvertex] (0) at (-2.5, 2.25) {};
		\node [style=blackvertex] (1) at (-1.75, 1) {};
		\node [style=blackvertex] (2) at (-1.75, -0.25) {};
		\node [style=blackvertex] (3) at (-2.5, -1.75) {};
		\node [style=blackvertex] (4) at (2.25, 3.25) {};
		\node [style=blackvertex] (5) at (4.25, 3) {};
		\node [style=blackvertex] (6) at (3.5, 0.75) {};
		\node [style=blackvertex] (7) at (1.75, 1.75) {};
		\node [style=blackvertex] (8) at (0.75, 0.5) {};
		\node [style=blackvertex] (9) at (0, -1.25) {};
		\node [style=blackvertex] (10) at (2.25, -1.25) {};
		\node [style=whitevertex] (11) at (-3.75, 1.5) {};
		\node [style=whitevertex] (12) at (-3, 0.5) {};
		\node [style=whitevertex] (13) at (-4, -1) {};
		\node [style=whitevertex] (14) at (-4.5, 0.25) {};
		\node [style=whitevertex] (15) at (-1, -1) {};
		\node [style=whitevertex] (16) at (-0.75, -2.25) {};
		\node [style=whitevertex] (17) at (1.25, -2.25) {};
		\node [style=whitevertex] (18) at (1, -0.75) {};
		\node [style=whitevertex] (19) at (2.75, 2.5) {};
		\node [style=whitevertex] (20) at (3.25, 1.75) {};
		\node [style=whitevertex] (21) at (-0.75, 0.5) {};
		\node [style=whitevertex] (22) at (-1.25, 2.75) {};
		\node [style=whitevertex] (23) at (0, 1.5) {};
		\node [style=whitevertex] (24) at (1, 3.25) {};
		\node [style=whitevertex] (25) at (2, 0.5) {};
		\node [style=whitevertex] (26) at (3.25, -0.25) {};
		\node [style=whitevertex] (27) at (3.25, -1.5) {};
		\node [style=whitevertex] (28) at (5, 0) {};
		\node [style=whitevertex] (29) at (5, 1.5) {};
		\node [style=whitevertex] (30) at (0, 3) {};

		\draw (11) to (14);
		\draw (14) to (13);
		\draw (0) to (11);
		\draw (14) to (12);
		\draw (11) to (12);
		\draw (12) to (13);
		\draw [style=red] (0) to (1);
		\draw [style=red] (1) to (2);
		\draw [style=red] (2) to (3);
		\draw [style=red] (8) to (9);
		\draw [style=red] (8) to (10);
		\draw [style=red] (10) to (9);
		\draw [style=red] (4) to (7);
		\draw [style=red] (7) to (6);
		\draw [style=red] (5) to (6);
		\draw [style=red] (4) to (5);
		\draw (0) to (22);
		\draw (22) to (30);
		\draw (30) to (24);
		\draw (24) to (4);
		\draw (24) to (7);
		\draw (30) to (23);
		\draw (22) to (23);
		\draw (22) to (1);
		\draw (1) to (23);
		\draw (30) to (7);
		\draw (7) to (23);
		\draw (4) to (19);
		\draw (19) to (7);
		\draw (7) to (20);
		\draw (19) to (20);
		\draw (19) to (5);
		\draw (5) to (20);
		\draw (20) to (6);
		\draw (5) to (29);
		\draw (29) to (6);
		\draw (6) to (28);
		\draw (29) to (28);
		\draw (7) to (25);
		\draw (25) to (26);
		\draw (25) to (6);
		\draw (6) to (26);
		\draw (26) to (28);
		\draw (28) to (27);
		\draw (26) to (27);
		\draw (10) to (26);
		\draw (10) to (27);
		\draw (27) to (17);
		\draw (9) to (17);
		\draw (10) to (17);
		\draw (18) to (9);
		\draw (8) to (18);
		\draw (18) to (10);
		\draw (25) to (10);
		\draw (8) to (25);
		\draw (7) to (8);
		\draw (23) to (8);
		\draw (23) to (21);
		\draw (21) to (8);
		\draw (1) to (21);
		\draw (21) to (2);
		\draw (21) to (15);
		\draw (15) to (3);
		\draw (3) to (16);
		\draw (2) to (15);
		\draw (21) to (9);
		\draw (15) to (9);
		\draw (15) to (16);
		\draw (16) to (9);
		\draw (16) to (17);
		\draw (0) to (12);
		\draw (1) to (12);
		\draw (12) to (2);
		\draw (12) to (3);
		\draw (13) to (3);

\end{tikzpicture}

%% file: faces_meet_on_edge.tex
\begin{tikzpicture}
	[scale=0.65,inner sep=0.8mm, 
	blackvertex/.style={circle,thick,draw, fill=black},
	whitevertex/.style={circle,thick,draw, fill=white},
	none/.style={},
	red/.style={line width=1.5pt}]
		\node [style=whitevertex, label={right:$u$}] (0) at (0, 1) {};
		\node [style=blackvertex, label={right:$v$}] (1) at (0, -1) {};
		\node [style=blackvertex, label={below:$s$}] (2) at (-2, 2) {};
		\node [style=blackvertex] (3) at (0, 3) {};
		\node [style=blackvertex, label={below:$y$}] (4) at (2, 2) {};
		\node [style=whitevertex, label={above:$t$}] (5) at (-2, -2) {};
		\node [style=whitevertex, label={above:$z$}] (6) at (2, -2) {};
		\node [style=none] (7) at (-2.5, 0) {$f_W$};
		\node [style=none] (8) at (2.5, 0) {$f_X$};
		\node [style=whitevertex, label={right:$u$}] (9) at (10.5, 0) {};
		\node [style=blackvertex, label={below:$s$}] (11) at (8.5, 2) {};
		\node [style=blackvertex] (12) at (10.5, 2.25) {};
		\node [style=blackvertex, label={below:$y$}] (13) at (12.5, 2) {};
		\node [style=blackvertex, label={above:$t$}] (14) at (8.5, -2) {};
		\node [style=blackvertex, label={above:$z$}] (15) at (12.5, -2) {};
		\node [style=none] (16) at (8, 0) {$f_W$};
		\node [style=none] (17) at (13, 0) {$f_X$};
		\node [style=blackvertex] (18) at (9.75, -2.5) {};
		\node [style=blackvertex] (19) at (11.25, -2.5) {};
		\node [style=none] (20) at (10.5, 3) {$P_1$};
		\node [style=none] (21) at (15.5, 0) {$Q_1$};
		\node [style=none] (22) at (10.5, -3.25) {$P_2$};
		\node [style=none] (23) at (5.5, 0) {$Q_2$};

		\draw (0) to (1);
		\draw (2) to (0);
		\draw (3) to (0);
		\draw (0) to (4);
		\draw [style=red] (1) to (5);
		\draw [style=red] (1) to (6);
		\draw [style=red, bend left] (2) to (3);
		\draw [style=red, bend left] (3) to (4);
		\draw [style=red, in=210, out=-210, looseness=2.25] (2) to (5);
		\draw [style=red, bend left=120, looseness=2.25] (4) to (6);
		\draw (11) to (9);
		\draw (12) to (9);
		\draw (9) to (13);
		\draw [style=red, bend left] (11) to (12);
		\draw [style=red, bend left] (12) to (13);
		\draw [style=red, in=210, out=-210, looseness=2.25] (11) to (14);
		\draw [style=red, bend left=120, looseness=2.25] (13) to (15);
		\draw (9) to (18);
		\draw (9) to (19);
		\draw [style=red, bend right=15] (14) to (18);
		\draw [style=red, bend right=15] (18) to (19);
		\draw [style=red, bend right=15] (19) to (15);
		\draw (9) to (14);
		\draw (9) to (15);

\end{tikzpicture}

%% file: 3conn_cleavable_inter_wells.tex
\begin{tikzpicture}
	[scale=0.6,inner sep=0.8mm, 
	blackvertex/.style={circle,thick,draw, fill=black},
	whitevertex/.style={circle,thick,draw, fill=white},
	none/.style={},
	red/.style={line width=2pt}]
		\node [style=blackvertex] (0) at (0, 0) {};
		\node [style=whitevertex] (1) at (-1, 2) {};
		\node [style=whitevertex] (2) at (1, 2) {};
		\node [style=whitevertex] (3) at (-2, 1) {};
		\node [style=whitevertex] (4) at (2, 1) {};
		\node [style=whitevertex] (5) at (-1, -2) {};
		\node [style=whitevertex] (6) at (1, -2) {};
		\node [style=whitevertex] (7) at (-2, -1) {};
		\node [style=whitevertex] (8) at (2, -1) {};

		\draw (1) to (0);
		\draw (3) to (0);
		\draw (7) to (0);
		\draw (0) to (5);
		\draw (0) to (6);
		\draw (0) to (8);
		\draw (4) to (0);
		\draw (0) to (2);
		\draw (1) to (2);
		\draw (2) to (4);
		\draw (4) to (8);
		\draw (8) to (6);
		\draw (6) to (5);
		\draw (5) to (7);
		\draw (7) to (3);
		\draw (3) to (1);
		\draw [bend right=90, looseness=1.75] (1) to (5);
		\draw [bend left=90, looseness=1.75] (2) to (6);

\end{tikzpicture}

%% file: cross_split_touch.tex
\begin{tikzpicture}
	[scale=0.7,inner sep=0.8mm, 
	blackvertex/.style={circle,thick,draw, fill=black},
	whitevertex/.style={circle,thick,draw, fill=white},
	greyvertex/.style={circle,thick,draw, fill=black!30},
	lightvertex/.style={star,draw, star points = 5},
	none/.style={},
	red/.style={line width=2pt},
	lightedge/.style={line width=1pt, color=black!50, dotted}]
		\node [style=greyvertex] (1) at (-4.5, 1.5) {};
		\node [style=greyvertex] (2) at (-3.25, 2.5) {};
		\node [style=greyvertex] (3) at (-2, 3) {};
		\node [style=greyvertex] (4) at (-0.5, 3.25) {};
		\node [style=greyvertex] (5) at (1, 3.25) {};
		\node [style=greyvertex] (6) at (2.75, 3) {};
		\node [style=greyvertex] (7) at (3.75, 2.25) {};
		\node [style=greyvertex] (8) at (4.5, 1.25) {};
		\node [style=greyvertex, label={left:$s$}] (10) at (-5, 0) {};
		\node [style=greyvertex] (11) at (-4.5, -1.5) {};
		\node [style=greyvertex] (12) at (-3.25, -2.5) {};
		\node [style=greyvertex] (13) at (-2, -3) {};
		\node [style=greyvertex] (14) at (-0.5, -3.25) {};
		\node [style=greyvertex] (15) at (1, -3.25) {};
		\node [style=greyvertex] (16) at (2.5, -3) {};
		\node [style=greyvertex, shape=star, star points = 5] (17) at (3.75, -2.25) {};
		\node [style=greyvertex] (18) at (4.5, -1.25) {};
		\node [style=greyvertex, label={right:$t$}] (19) at (5, 0) {};
		\node [style=whitevertex] (20) at (-3, 1.5) {};
		\node [style=whitevertex] (21) at (-1, 2) {};
		\node [style=whitevertex] (22) at (-3.5, 0.25) {};
		\node [style=whitevertex] (23) at (-1, 0.25) {};
		\node [style=whitevertex] (24) at (-1.5, -1) {};
		\node [style=whitevertex] (25) at (-3.5, -1) {};
		\node [style=lightvertex] (26) at (-2.25, 0.25) {};
		\node [style=whitevertex] (27) at (0.75, 2.25) {};
		\node [style=whitevertex] (28) at (0.75, 1) {};
		\node [style=whitevertex] (29) at (2.25, 0) {};
		\node [style=whitevertex] (30) at (3.75, 0.5) {};
		\node [style=whitevertex] (31) at (2.5, 2.25) {};
		\node [style=lightvertex] (32) at (2.25, 1.25) {};
		\node [style=whitevertex] (36) at (-1, -1.75) {};
		\node [style=whitevertex] (37) at (1, -1.75) {};
		\node [style=whitevertex] (38) at (0.25, -0.75) {};
		\node [style=lightvertex] (39) at (0, -1.75) {};
		\node [style=whitevertex] (40) at (2.75, -1.5) {};
		\node [style=whitevertex] (41) at (3.75, -3.5) {};
		\node [style=whitevertex] (42) at (5, -2.5) {};

		\draw [style=red] (3) to (20);
		\draw [style=red] (20) to (22);
		\draw [style=red] (22) to (25);
		\draw [style=red] (25) to (12);
		\draw [style=red] (12) to (24);
		\draw [style=red] (24) to (23);
		\draw [style=red] (23) to (21);
		\draw [style=red] (21) to (3);
		\draw [style=lightedge] (20) to (26);
		\draw [style=lightedge] (3) to (26);
		\draw [style=lightedge] (26) to (21);
		\draw [style=lightedge] (26) to (23);
		\draw [style=lightedge] (26) to (24);
		\draw [style=lightedge] (26) to (12);
		\draw [style=lightedge] (22) to (26);
		\draw [style=lightedge] (25) to (26);
		\draw (10) to (1);
		\draw (1) to (2);
		\draw (2) to (3);
		\draw (2) to (20);
		\draw (1) to (20);
		\draw (1) to (22);
		\draw (22) to (10);
		\draw (22) to (11);
		\draw (25) to (11);
		\draw (10) to (11);
		\draw (11) to (12);
		\draw (12) to (13);
		\draw (24) to (13);
		\draw (4) to (21);
		\draw (3) to (4);
		\draw [style=red] (5) to (27);
		\draw [style=red] (27) to (28);
		\draw [style=red] (28) to (29);
		\draw [style=red] (29) to (30);
		\draw [style=red] (30) to (8);
		\draw [style=red] (7) to (8);
		\draw [style=red] (31) to (7);
		\draw [style=red] (5) to (31);
		\draw (5) to (6);
		\draw (6) to (31);
		\draw (6) to (7);
		\draw [style=lightedge] (5) to (32);
		\draw [style=lightedge] (27) to (32);
		\draw [style=lightedge] (31) to (32);
		\draw [style=lightedge] (32) to (7);
		\draw [style=lightedge] (32) to (8);
		\draw [style=lightedge] (32) to (30);
		\draw [style=lightedge] (32) to (29);
		\draw [style=lightedge] (32) to (28);
		\draw (15) to (16);
		\draw [style=lightedge] (16) to (17);
		\draw [style=red] (36) to (14);
		\draw [style=red] (14) to (37);
		\draw [style=red] (38) to (37);
		\draw [style=red] (38) to (36);
		\draw [style=lightedge] (36) to (39);
		\draw [style=lightedge] (38) to (39);
		\draw [style=lightedge] (39) to (37);
		\draw [style=lightedge] (39) to (14);
		\draw (8) to (19);
		\draw (19) to (18);
		\draw [style=lightedge] (18) to (17);
		\draw [style=red] (16) to (40);
		\draw [style=red] (40) to (18);
		\draw [style=red] (18) to (42);
		\draw [style=red] (42) to (41);
		\draw [style=red] (16) to (41);
		\draw [style=lightedge] (40) to (17);
		\draw [style=lightedge] (17) to (41);
		\draw [style=lightedge] (17) to (42);
		\draw (13) to (14);
		\draw (14) to (15);
		\draw (36) to (13);
		\draw (24) to (36);
		\draw (24) to (38);
		\draw (23) to (38);
		\draw (4) to (5);
		\draw (30) to (19);
		\draw (30) to (18);
		\draw (30) to (40);
		\draw (29) to (40);
		\draw (40) to (37);
		\draw (37) to (15);
		\draw (37) to (16);
		\draw (29) to (37);
		\draw (38) to (29);
		\draw (28) to (38);
		\draw (23) to (28);
		\draw (21) to (28);
		\draw (21) to (27);
		\draw (4) to (27);

\end{tikzpicture}

%% file: path_detour.tex
\begin{tikzpicture}
	[scale=0.7,inner sep=0.8mm, 
	blackvertex/.style={circle,thick,draw, fill=black},
	whitevertex/.style={circle,thick,draw, fill=white},
	greyvertex/.style={circle,thick,draw, fill=black!30},
	star/.style={shape=star, star points = 5, thick, draw, fill=white},
	none/.style={},
	red/.style={line width=2pt},
	dashed_thick/.style={line width=2pt, dotted},
	lightedge/.style={line width=1pt, color=black!50, dotted}]
	
	\pgfdeclarelayer{bg} 
	\pgfdeclarelayer{nodelayer} 
	\pgfdeclarelayer{edgelayer}  
	\pgfsetlayers{bg, edgelayer, nodelayer, main} 
	
	\begin{pgfonlayer}{nodelayer}
		\node [style=greyvertex] (1) at (-4, 3) {};
		\node [style=greyvertex] (2) at (-2.5, 3.5) {};
		\node [style=greyvertex] (3) at (-1, 3.75) {};
		\node [style=greyvertex] (4) at (0.75, 3.75) {};
		\node [style=greyvertex] (5) at (2.25, 3.5) {};
		\node [style=greyvertex] (6) at (3.5, 2.75) {};
		\node [style=whitevertex] (7) at (-1.75, 2) {};
		\node [style=whitevertex] (9) at (-1, 0.5) {};
		\node [style=whitevertex] (10) at (1, 0.5) {};
		\node [style=whitevertex] (11) at (1.5, 2) {};
		\node [style=star] (12) at (0, 2.25) {};
		\node [style=whitevertex] (16) at (3.25, 1) {};
		\node [style=whitevertex] (17) at (-3.5, 1) {};
		\node [style=greyvertex] (18) at (6.25, 3) {};
		\node [style=greyvertex] (19) at (7.75, 3.5) {};
		\node [style=greyvertex] (20) at (9.25, 3.75) {};
		\node [style=greyvertex] (21) at (11, 3.75) {};
		\node [style=greyvertex] (22) at (12.5, 3.5) {};
		\node [style=greyvertex] (23) at (13.75, 2.75) {};
		\node [style=whitevertex] (24) at (8.5, 2) {};
		\node [style=whitevertex] (25) at (9.25, 0.5) {};
		\node [style=whitevertex] (26) at (11.25, 0.5) {};
		\node [style=whitevertex] (27) at (11.75, 2) {};
		\node [style=star] (28) at (10.25, 2.25) {};
		\node [style=whitevertex] (29) at (13.5, 1) {};
		\node [style=whitevertex] (30) at (6.75, 1) {};
		\node [style=none] (31) at (5, 1.5) {\Huge $\mapsto$};
		\node [style=none] (32) at (-2.75, 0.75) {$P$};
		\node [style=none] (33) at (7.75, 0.75) {$P'$};
	\end{pgfonlayer}
	\begin{pgfonlayer}{edgelayer}
		\draw (1) to (2);
		\draw (2) to (3);
		\draw (3) to (4);
		\draw (4) to (5);
		\draw (5) to (6);
		\draw [style=lightedge] (3) to (12);
		\draw [style=lightedge] (4) to (12);
		\draw [style=lightedge] (12) to (10);
		\draw [style=lightedge] (12) to (9);
		\draw (3) to (7);
		\draw (7) to (9);
		\draw (9) to (10);
		\draw (10) to (11);
		\draw (11) to (4);
		\draw [style={dashed_thick}, bend left=15, looseness=0.75] (7) to (12);
		\draw [style={dashed_thick}, bend left=15, looseness=0.75] (12) to (11);
		\draw [style=red, bend left=15, looseness=0.75] (17) to (7);
		\draw [style=red, in=135, out=-15, looseness=0.75] (11) to (16);
		\draw (18) to (19);
		\draw (19) to (20);
		\draw (20) to (21);
		\draw (21) to (22);
		\draw (22) to (23);
		\draw [style=lightedge] (20) to (28);
		\draw [style=lightedge] (21) to (28);
		\draw [style=lightedge] (28) to (26);
		\draw [style=lightedge] (28) to (25);
		\draw (20) to (24);
		\draw [style=red] (24) to (25);
		\draw [style=red] (25) to (26);
		\draw [style=red] (26) to (27);
		\draw (27) to (21);
		\draw [style=lightedge, bend left=15, looseness=0.75] (24) to (28);
		\draw [style=lightedge, bend left=15, looseness=0.75] (28) to (27);
		\draw [style=red, bend left=15, looseness=0.75] (30) to (24);
		\draw [style=red, in=135, out=-15, looseness=0.75] (27) to (29);
	\end{pgfonlayer}
\end{tikzpicture}

%% file: main_theorem_setup.tex
\begin{tikzpicture}
	[scale=0.8,inner sep=0.8mm, 
	blackvertex/.style={circle,thick,draw, fill=black},
	whitevertex/.style={circle,thick,draw, fill=white},
	greyvertex/.style={circle,thick,draw, fill=black!30},
	lightvertex/.style={circle,draw, fill=black!15},
	none/.style={},
	red/.style={line width=2pt},
	lightedge/.style={line width=1pt, color=black!50, dotted},
	star/.style={shape=star, star points = 5, thick, draw, fill=white},
	arrow/.style={->, line width = 0.5pt}]
	
	\pgfdeclarelayer{bg}    
	\pgfsetlayers{bg,main}

		\node [style=whitevertex] (1) at (-3.25, 1.25) {};
		\node [style=whitevertex, label={above:$x$}] (2) at (-2, 2) {};
		\node [style=whitevertex] (3) at (-0.5, 2.5) {};
		\node [style=whitevertex] (4) at (1, 2.5) {};
		\node [style=whitevertex, label={above:$y$}] (5) at (2.5, 2) {};
		\node [style=whitevertex] (6) at (3.75, 1.25) {};
		\node [style=greyvertex, label={right:$t$}] (8) at (4.5, 0) {};
		\node [style=whitevertex] (9) at (3.75, -1) {};
		\node [style=whitevertex] (10) at (2.5, -1.75) {};
		\node [style=whitevertex] (11) at (1, -2) {};
		\node [style=whitevertex] (12) at (-0.5, -2) {};
		\node [style=whitevertex] (13) at (-2, -1.75) {};
		\node [style=whitevertex] (14) at (-3.25, -1) {};
		\node [style=greyvertex, label={left:$s$}] (15) at (-4, 0) {};
		\node [style=whitevertex] (17) at (-1.75, 0.75) {};
		\node [style=whitevertex] (18) at (-0.5, 0) {};
		\node [style=whitevertex] (19) at (2.25, 0.75) {};
		\node [style=whitevertex] (20) at (0.25, 1.75) {};
		\node [style=star] (21) at (0.25, 0.925) {};
		\node [style=whitevertex] (22) at (1, 0) {};
		\node [style=whitevertex] (23) at (-2.75, 0.25) {};
		\node [style=whitevertex] (24) at (-1.5, -0.25) {};
		\node [style=whitevertex] (25) at (0.25, -0.75) {};
		\node [style=whitevertex] (26) at (1.75, -0.75) {};
		\node [style=whitevertex] (27) at (2.95, 0.175) {};
		\node [style=whitevertex] (28) at (-1, -1.25) {};
		\node [style=whitevertex] (29) at (-2.325, -0.8) {};
		\node [style=whitevertex] (30) at (2.75, -0.75) {};
		\node [style=whitevertex] (31) at (-4, -1.25) {};
		\node [style=whitevertex] (32) at (-3, -2.25) {};
		\node [style=whitevertex] (34) at (-1.5, -2.75) {};
		\node [style=whitevertex] (36) at (0, -3) {};
		\node [style=whitevertex] (37) at (1.5, -3) {};
		\node [style=whitevertex] (38) at (3, -2.75) {};
		\node [style=whitevertex] (39) at (4, -2) {};
		\node [style=whitevertex] (40) at (4.5, -1) {};
		\node [style=whitevertex] (41) at (2, -2.4) {};
		\node [style=none] (42) at (-4.75, 1.5) {$P_1$};
		\node [style=none] (43) at (-3.25, 0.15) {};
		\node [style=none] (44) at (3, 1.7) {};
		\node [style=none] (45) at (4.75, 1.75) {$Q_1$};
		\node [style=none] (46) at (-5, -2) {$P_2$};
		\node [style=none] (47) at (-3.5, -1.75) {};
		\node [style=none] (48) at (3.2, -1.35) {};
		\node [style=none] (49) at (4.75, -1.75) {$Q_2$};
	
\draw (3) to (20);
\draw (20) to (4);
\draw [style=red] (15) to (1);
\draw [style=red] (1) to (2);
\draw [style=red] (2) to (3);
\draw [style=red] (3) to (4);
\draw [style=red] (4) to (5);
\draw [style=red] (5) to (6);
\draw [style=red] (6) to (8);
\draw [style=red] (15) to (14);
\draw [style=red] (14) to (13);
\draw [style=red] (13) to (12);
\draw [style=red] (12) to (11);
\draw [style=red] (11) to (10);
\draw [style=red] (10) to (9);
\draw [style=red] (9) to (8);
\draw [style=lightedge] (2) to (21);
\draw [style=lightedge] (20) to (21);
\draw [style=lightedge] (21) to (5);
\draw [style=lightedge] (21) to (19);
\draw [style=lightedge] (21) to (22);
\draw [style=lightedge] (21) to (18);
\draw [style=lightedge] (17) to (21);
\draw [style=red] (15) to (23);
\draw [style=red] (23) to (24);
\draw [style=red] (24) to (18);
\draw (22) to (26);
\draw (26) to (27);
\draw [style=red] (27) to (8);
\draw (17) to (24);
\draw (17) to (23);
\draw (1) to (23);
\draw (1) to (17);
\draw [style=red] (22) to (27);
\draw (19) to (27);
\draw (19) to (6);
\draw (6) to (27);
\draw (23) to (29);
\draw (23) to (14);
\draw (14) to (29);
\draw (29) to (13);
\draw (13) to (24);
\draw (24) to (29);
\draw (24) to (28);
\draw (28) to (18);
\draw (18) to (25);
\draw (25) to (22);
\draw (25) to (28);
\draw (28) to (12);
\draw (28) to (13);
\draw (12) to (25);
\draw (25) to (26);
\draw (26) to (11);
\draw (12) to (26);
\draw (26) to (10);
\draw (10) to (30);
\draw (30) to (27);
\draw (30) to (9);
\draw (30) to (8);
\draw (26) to (30);
\draw [style=red] (18) to (22);
\draw (2) to (17);
\draw (17) to (18);
\draw (22) to (19);
\draw (19) to (5);
\draw (5) to (20);
\draw (20) to (2);
\draw [style=red] (15) to (31);
\draw [style=red] (31) to (32);
\draw [style=red] (32) to (34);
\draw [style=red] (34) to (36);
\draw [style=red] (36) to (37);
\draw [style=red] (37) to (38);
\draw [style=red] (38) to (39);
\draw [style=red] (39) to (40);
\draw [style=red] (40) to (8);
\draw (31) to (14);
\draw (14) to (32);
\draw (32) to (13);
\draw (13) to (34);
\draw (13) to (36);
\draw (12) to (36);
\draw (36) to (11);
\draw (11) to (37);
\draw (11) to (41);
\draw (41) to (37);
\draw (41) to (38);
\draw (10) to (41);
\draw (10) to (38);
\draw (10) to (39);
\draw (9) to (39);
\draw (9) to (40);
\draw [style=arrow] (42) to (43);
\draw [style=arrow] (45) to (44);
\draw [style=arrow] (46) to (47);
\draw [style=arrow] (49) to (48);
	
\end{tikzpicture}

%% file: algo_steps.tex
\begin{tikzpicture}
	[scale=0.6,inner sep=0.6mm, 
	blackvertex/.style={circle,thick,draw, fill=black},
	whitevertex/.style={circle,thick,draw, fill=white},
	greyvertex/.style={circle,thick,draw, fill=black!30},
	star/.style={shape=star, star points = 5, thick, draw, fill=white},
	blackstar/.style={shape=star, star points = 5, thick, draw, fill=black},
	none/.style={},
	red/.style={line width=1.3pt},
	thickdashed/.style={line width=1.3pt, dotted},
	lightedge/.style={line width=1pt, color=black!50, dotted},
	alternatepath/.style={double},
	shaded/.style={},]
	
	\pgfdeclarelayer{bg} 
	\pgfdeclarelayer{nodelayer} 
	\pgfdeclarelayer{edgelayer}  
	\pgfsetlayers{bg, edgelayer, nodelayer, main} 
	
	\begin{pgfonlayer}{nodelayer}
		\node [style=blackvertex] (0) at (-1.5, -3.5) {};
		\node [style=whitevertex] (1) at (-1, -4.25) {};
		\node [style=whitevertex] (2) at (-0.25, -4.25) {};
		\node [style=whitevertex] (3) at (0.25, -3.5) {};
		\node [style=blackvertex] (4) at (-0.25, -2.75) {};
		\node [style=whitevertex] (5) at (-1, -2.75) {};
		\node [style=blackvertex] (6) at (-1.5, 3) {};
		\node [style=whitevertex] (7) at (-1, 2.25) {};
		\node [style=blackvertex] (8) at (-0.25, 2.25) {};
		\node [style=whitevertex] (9) at (0.25, 3) {};
		\node [style=whitevertex] (10) at (-0.25, 3.75) {};
		\node [style=whitevertex] (11) at (-1, 3.75) {};
		\node [style=whitevertex] (12) at (-1.75, -4.25) {};
		\node [style=blackvertex] (13) at (-1.75, -2.75) {};
		\node [style=whitevertex] (14) at (-0.75, -2) {};
		\node [style=whitevertex] (15) at (-1.5, -2) {};
		\node [style=blackvertex] (16) at (0, -2) {};
		\node [style=whitevertex] (17) at (0.5, -2.75) {};
		\node [style=whitevertex] (18) at (1, -3.5) {};
		\node [style=whitevertex] (19) at (0.75, -4.5) {};
		\node [style=whitevertex] (20) at (0.25, -5) {};
		\node [style=whitevertex] (21) at (-0.5, -5.25) {};
		\node [style=whitevertex] (22) at (-1.5, -5) {};
		\node [style=whitevertex] (23) at (-2.25, -3.25) {};
		\node [style=whitevertex] (24) at (-1, -1.25) {};
		\node [style=blackvertex] (25) at (0, -1) {};
		\node [style=whitevertex] (26) at (0.75, -1.75) {};
		\node [style=whitevertex] (27) at (1.25, -2.75) {};
		\node [style=whitevertex] (28) at (1.5, -1.5) {};
		\node [style=whitevertex] (29) at (1.25, -0.75) {};
		\node [style=blackvertex] (30) at (0.5, -0.25) {};
		\node [style=whitevertex] (31) at (-0.4, -0.175) {};
		\node [style=whitevertex] (32) at (-1.25, -0.5) {};
		\node [style=whitevertex] (33) at (-1.75, -1.25) {};
		\node [style=blackvertex] (34) at (-2.25, -2) {};
		\node [style=blackvertex] (35) at (-2.5, -1) {};
		\node [style=blackvertex] (36) at (-2.75, 0) {};
		\node [style=whitevertex] (37) at (-1, 0.5) {};
		\node [style=blackvertex] (38) at (0.25, 0.5) {};
		\node [style=whitevertex] (39) at (1.25, 0.25) {};
		\node [style=whitevertex] (40) at (1.5, 1.25) {};
		\node [style=whitevertex] (41) at (0.85, 1) {};
		\node [style=blackvertex] (42) at (0, 1.25) {};
		\node [style=whitevertex] (43) at (-1, 1.45) {};
		\node [style=blackvertex] (44) at (-2.5, 1) {};
		\node [style=blackvertex] (45) at (-2, 2) {};
		\node [style=whitevertex] (46) at (0.5, 2) {};
		\node [style=whitevertex] (47) at (1.25, 2.5) {};
		\node [style=whitevertex] (48) at (0.75, 3.5) {};
		\node [style=whitevertex] (49) at (1.75, 3.5) {};
		\node [style=whitevertex] (50) at (1.25, 4.25) {};
		\node [style=whitevertex] (51) at (0.25, 4) {};
		\node [style=whitevertex] (52) at (-0.5, 4.75) {};
		\node [style=whitevertex] (53) at (-1.5, 4.5) {};
		\node [style=whitevertex] (54) at (-2.025, 3.85) {};
		\node [style=whitevertex] (55) at (-2.5, 3) {};
		\node [style=whitevertex] (56) at (-3, 2) {};
		\node [style=whitevertex] (57) at (-1.7, 1.125) {};
		\node [style=whitevertex] (58) at (-3, -2.25) {};
		\node [style=blackstar] (59) at (-0.6, -3.5) {};
		\node [style=blackstar] (60) at (-0.6, 3) {};
		\node [style=star] (61) at (-1.9, 0) {};
		\node [style=whitevertex] (62) at (1.75, -0.25) {};
		\node [style=blackvertex] (63) at (5.5, -3.5) {};
		\node [style=whitevertex] (64) at (6, -4.25) {};
		\node [style=whitevertex] (65) at (6.75, -4.25) {};
		\node [style=whitevertex] (66) at (7.25, -3.5) {};
		\node [style=blackvertex] (67) at (6.75, -2.75) {};
		\node [style=whitevertex] (68) at (6, -2.75) {};
		\node [style=blackvertex] (69) at (5.5, 3) {};
		\node [style=whitevertex] (70) at (6, 2.25) {};
		\node [style=blackvertex] (71) at (6.75, 2.25) {};
		\node [style=whitevertex] (72) at (7.25, 3) {};
		\node [style=whitevertex] (73) at (6.75, 3.75) {};
		\node [style=whitevertex] (74) at (6, 3.75) {};
		\node [style=whitevertex] (75) at (5.25, -4.25) {};
		\node [style=blackvertex] (76) at (5.25, -2.75) {};
		\node [style=whitevertex] (77) at (6.25, -2) {};
		\node [style=whitevertex] (78) at (5.5, -2) {};
		\node [style=blackvertex] (79) at (7, -2) {};
		\node [style=whitevertex] (80) at (7.5, -2.75) {};
		\node [style=whitevertex] (81) at (8, -3.5) {};
		\node [style=whitevertex] (82) at (7.75, -4.5) {};
		\node [style=whitevertex] (83) at (7.25, -5) {};
		\node [style=whitevertex] (84) at (6.5, -5.25) {};
		\node [style=whitevertex] (85) at (5.5, -5) {};
		\node [style=whitevertex] (86) at (4.75, -3.25) {};
		\node [style=whitevertex] (87) at (6, -1.25) {};
		\node [style=blackvertex] (88) at (7, -1) {};
		\node [style=whitevertex] (89) at (7.75, -1.75) {};
		\node [style=whitevertex] (90) at (8.25, -2.75) {};
		\node [style=whitevertex] (91) at (8.5, -1.5) {};
		\node [style=whitevertex] (92) at (8.25, -0.75) {};
		\node [style=blackvertex] (93) at (7.5, -0.25) {};
		\node [style=whitevertex] (94) at (6.6, -0.175) {};
		\node [style=whitevertex] (95) at (5.75, -0.5) {};
		\node [style=whitevertex] (96) at (5.25, -1.25) {};
		\node [style=blackvertex] (97) at (4.75, -2) {};
		\node [style=blackvertex] (98) at (4.5, -1) {};
		\node [style=whitevertex] (99) at (4.25, 0) {};
		\node [style=whitevertex] (100) at (6, 0.5) {};
		\node [style=blackvertex] (101) at (7.25, 0.5) {};
		\node [style=whitevertex] (102) at (8.25, 0.25) {};
		\node [style=whitevertex] (103) at (8.5, 1.25) {};
		\node [style=whitevertex] (104) at (7.85, 1) {};
		\node [style=blackvertex] (105) at (7, 1.25) {};
		\node [style=whitevertex] (106) at (6, 1.45) {};
		\node [style=blackvertex] (107) at (4.5, 1) {};
		\node [style=blackvertex] (108) at (5, 2) {};
		\node [style=whitevertex] (109) at (7.5, 2) {};
		\node [style=whitevertex] (110) at (8.25, 2.5) {};
		\node [style=whitevertex] (111) at (7.75, 3.5) {};
		\node [style=whitevertex] (112) at (8.75, 3.5) {};
		\node [style=whitevertex] (113) at (8.25, 4.25) {};
		\node [style=whitevertex] (114) at (7.25, 4) {};
		\node [style=whitevertex] (115) at (6.5, 4.75) {};
		\node [style=whitevertex] (116) at (5.5, 4.5) {};
		\node [style=whitevertex] (117) at (4.975, 3.85) {};
		\node [style=whitevertex] (118) at (4.5, 3) {};
		\node [style=whitevertex] (119) at (4, 2) {};
		\node [style=whitevertex] (120) at (5.3, 1.125) {};
		\node [style=whitevertex] (121) at (4, -2.25) {};
		\node [style=blackstar] (122) at (6.4, -3.5) {};
		\node [style=blackstar] (123) at (6.4, 3) {};
		\node [style=blackstar] (124) at (5.1, 0) {};
		\node [style=whitevertex] (125) at (8.75, -0.25) {};
		\node [style=blackvertex] (126) at (12.5, -3.5) {};
		\node [style=whitevertex] (127) at (13, -4.25) {};
		\node [style=whitevertex] (128) at (13.75, -4.25) {};
		\node [style=whitevertex] (129) at (14.25, -3.5) {};
		\node [style=blackvertex] (130) at (13.75, -2.75) {};
		\node [style=greyvertex] (131) at (13, -2.75) {};
		\node [style=blackvertex] (132) at (12.5, 3) {};
		\node [style=greyvertex] (133) at (13, 2.25) {};
		\node [style=blackvertex] (134) at (13.75, 2.25) {};
		\node [style=whitevertex] (135) at (14.25, 3) {};
		\node [style=whitevertex] (136) at (13.75, 3.75) {};
		\node [style=whitevertex] (137) at (13, 3.75) {};
		\node [style=whitevertex] (138) at (12.25, -4.25) {};
		\node [style=blackvertex] (139) at (12.25, -2.75) {};
		\node [style=greyvertex] (140) at (13.25, -2) {};
		\node [style=greyvertex] (141) at (12.5, -2) {};
		\node [style=blackvertex] (142) at (14, -2) {};
		\node [style=whitevertex] (143) at (14.5, -2.75) {};
		\node [style=whitevertex] (144) at (15, -3.5) {};
		\node [style=whitevertex] (145) at (14.75, -4.5) {};
		\node [style=whitevertex] (146) at (14.25, -5) {};
		\node [style=whitevertex] (147) at (13.5, -5.25) {};
		\node [style=whitevertex] (148) at (12.5, -5) {};
		\node [style=whitevertex] (149) at (11.75, -3.25) {};
		\node [style=greyvertex] (150) at (13, -1.25) {};
		\node [style=blackvertex] (151) at (14, -1) {};
		\node [style=whitevertex] (152) at (14.75, -1.75) {};
		\node [style=whitevertex] (153) at (15.25, -2.75) {};
		\node [style=whitevertex] (154) at (15.5, -1.5) {};
		\node [style=whitevertex] (155) at (15.25, -0.75) {};
		\node [style=blackvertex] (156) at (14.5, -0.25) {};
		\node [style=greyvertex] (157) at (13.6, -0.175) {};
		\node [style=greyvertex] (158) at (12.75, -0.5) {};
		\node [style=greyvertex] (159) at (12.25, -1.25) {};
		\node [style=blackvertex] (160) at (11.75, -2) {};
		\node [style=blackvertex] (161) at (11.5, -1) {};
		\node [style=whitevertex] (162) at (11.25, 0) {};
		\node [style=greyvertex] (163) at (13, 0.5) {};
		\node [style=blackvertex] (164) at (14.25, 0.5) {};
		\node [style=whitevertex] (165) at (15.25, 0.25) {};
		\node [style=whitevertex] (166) at (15.5, 1.25) {};
		\node [style=whitevertex] (167) at (14.85, 1) {};
		\node [style=blackvertex] (168) at (14, 1.25) {};
		\node [style=greyvertex] (169) at (13, 1.45) {};
		\node [style=blackvertex] (170) at (11.5, 1) {};
		\node [style=blackvertex] (171) at (12, 2) {};
		\node [style=whitevertex] (172) at (14.5, 2) {};
		\node [style=whitevertex] (173) at (15.25, 2.5) {};
		\node [style=whitevertex] (174) at (14.75, 3.5) {};
		\node [style=whitevertex] (175) at (15.75, 3.5) {};
		\node [style=whitevertex] (176) at (15.25, 4.25) {};
		\node [style=whitevertex] (177) at (14.25, 4) {};
		\node [style=whitevertex] (178) at (13.5, 4.75) {};
		\node [style=whitevertex] (179) at (12.5, 4.5) {};
		\node [style=whitevertex] (180) at (11.975, 3.85) {};
		\node [style=whitevertex] (181) at (11.5, 3) {};
		\node [style=whitevertex] (182) at (11, 2) {};
		\node [style=greyvertex] (183) at (12.3, 1.125) {};
		\node [style=whitevertex] (184) at (11, -2.25) {};
		\node [style=whitevertex] (188) at (15.75, -0.25) {};
		\node [style=none] (189) at (-0.5, -7) {$C$ in $G^\nabla$};
		\node [style=none] (192) at (6.5, -7) {$D$ in $G^\nabla$};
		\node [style=none] (193) at (13.5, -7) {Minimal separator in $G$};
		\node [style=none] (194) at (2.75, 0) {$\mapsto$};
		\node [style=none] (195) at (9.75, 0) {$\mapsto$};
	\end{pgfonlayer}
	\begin{pgfonlayer}{edgelayer}
		\draw [fill=black!10] (57.center)
		to (44.center)
		to (36.center)
		to (35.center)
		to (33.center)
		to (32.center)
		to (37.center)
		to cycle;
		\draw [fill=black!10] (6.center)
		to (7.center)
		to (8.center)
		to (9.center)
		to (10.center)
		to (11.center)
		to cycle;
		\draw [fill=black!10] (2.center)
		to (3.center)
		to (4.center)
		to (5.center)
		to (0.center)
		to (1.center)
		to cycle;
		\draw (0) to (1);
		\draw (1) to (2);
		\draw (2) to (3);
		\draw (3) to (4);
		\draw (4) to (5);
		\draw (5) to (0);
		\draw (6) to (7);
		\draw (7) to (8);
		\draw (8) to (9);
		\draw (9) to (10);
		\draw (10) to (11);
		\draw (11) to (6);
		\draw (0) to (13);
		\draw (13) to (34);
		\draw (34) to (35);
		\draw (35) to (36);
		\draw (36) to (44);
		\draw (44) to (45);
		\draw (45) to (6);
		\draw [style=alternatepath] (5) to (14);
		\draw [style=alternatepath] (14) to (24);
		\draw [style=alternatepath] (24) to (32);
		\draw [style=alternatepath] (32) to (37);
		\draw [style=alternatepath] (37) to (43);
		\draw [style=alternatepath] (43) to (7);
		\draw (4) to (16);
		\draw (16) to (25);
		\draw (25) to (30);
		\draw (30) to (38);
		\draw (38) to (42);
		\draw (42) to (8);
		\draw [style=alternatepath] (3) to (17);
		\draw (17) to (26);
		\draw (26) to (29);
		\draw [style=alternatepath] (29) to (39);
		\draw [style=alternatepath] (39) to (40);
		\draw [style=alternatepath] (40) to (47);
		\draw [style=alternatepath] (47) to (9);
		\draw [style=lightedge] (7) to (60);
		\draw [style=thickdashed] (6) to (60);
		\draw [style=lightedge] (60) to (11);
		\draw [style=lightedge] (60) to (10);
		\draw [style=lightedge] (60) to (9);
		\draw [style=thickdashed] (60) to (8);
		\draw [style=thickdashed] (0) to (59);
		\draw [style=lightedge] (1) to (59);
		\draw [style=lightedge] (59) to (5);
		\draw [style=thickdashed] (59) to (4);
		\draw [style=lightedge] (59) to (3);
		\draw [style=lightedge] (59) to (2);
		\draw (35) to (33);
		\draw (33) to (32);
		\draw (37) to (57);
		\draw (44) to (57);
		\draw [style=lightedge] (35) to (61);
		\draw [style=lightedge] (36) to (61);
		\draw [style=lightedge] (61) to (44);
		\draw [style=lightedge] (61) to (57);
		\draw [style=lightedge] (61) to (37);
		\draw [style=lightedge] (61) to (32);
		\draw [style=lightedge] (33) to (61);
		\draw [style=red] (0) to (13);
		\draw [style=red] (13) to (34);
		\draw [style=red] (34) to (35);
		\draw [style=red] (35) to (36);
		\draw [style=red] (36) to (44);
		\draw [style=red] (44) to (45);
		\draw [style=red] (45) to (6);
		\draw [style=red] (4) to (16);
		\draw [style=red] (16) to (25);
		\draw [style=red] (25) to (30);
		\draw [style=red] (30) to (38);
		\draw [style=red] (38) to (42);
		\draw [style=red] (42) to (8);
		\draw (13) to (5);
		\draw (5) to (15);
		\draw (15) to (14);
		\draw (34) to (15);
		\draw (13) to (15);
		\draw (15) to (33);
		\draw (34) to (33);
		\draw (33) to (14);
		\draw (33) to (24);
		\draw (45) to (57);
		\draw (57) to (43);
		\draw (45) to (43);
		\draw (45) to (7);
		\draw (44) to (56);
		\draw (56) to (45);
		\draw (56) to (55);
		\draw (55) to (6);
		\draw (55) to (54);
		\draw (54) to (11);
		\draw (54) to (53);
		\draw (36) to (56);
		\draw (45) to (55);
		\draw (6) to (54);
		\draw (11) to (53);
		\draw (53) to (52);
		\draw (53) to (10);
		\draw (10) to (52);
		\draw (52) to (51);
		\draw (10) to (51);
		\draw (52) to (50);
		\draw (51) to (50);
		\draw (50) to (49);
		\draw (48) to (50);
		\draw (51) to (48);
		\draw (9) to (51);
		\draw (9) to (48);
		\draw (47) to (48);
		\draw (47) to (49);
		\draw (48) to (49);
		\draw (40) to (49);
		\draw (58) to (36);
		\draw (58) to (35);
		\draw (58) to (34);
		\draw (58) to (23);
		\draw (23) to (34);
		\draw (23) to (13);
		\draw (23) to (12);
		\draw (12) to (0);
		\draw (23) to (0);
		\draw (12) to (1);
		\draw (12) to (22);
		\draw (22) to (1);
		\draw (22) to (21);
		\draw (21) to (20);
		\draw [bend left=15] (22) to (58);
		\draw (21) to (1);
		\draw (21) to (2);
		\draw (20) to (2);
		\draw (20) to (19);
		\draw (20) to (3);
		\draw (19) to (18);
		\draw (19) to (3);
		\draw (3) to (18);
		\draw (18) to (17);
		\draw (18) to (27);
		\draw [style=alternatepath] (17) to (27);
		\draw (27) to (26);
		\draw [style=alternatepath] (27) to (28);
		\draw (26) to (28);
		\draw [style=alternatepath] (28) to (29);
		\draw [bend right] (19) to (28);
		\draw (14) to (4);
		\draw (4) to (17);
		\draw (17) to (16);
		\draw (14) to (16);
		\draw (16) to (26);
		\draw (39) to (41);
		\draw (41) to (40);
		\draw (41) to (46);
		\draw (46) to (40);
		\draw (46) to (47);
		\draw (46) to (9);
		\draw (8) to (46);
		\draw (30) to (29);
		\draw (30) to (39);
		\draw (38) to (39);
		\draw (38) to (41);
		\draw (42) to (41);
		\draw (42) to (46);
		\draw (25) to (26);
		\draw (25) to (29);
		\draw (14) to (25);
		\draw (24) to (25);
		\draw (32) to (25);
		\draw (25) to (31);
		\draw (31) to (30);
		\draw (32) to (31);
		\draw (31) to (37);
		\draw (31) to (38);
		\draw (37) to (38);
		\draw (37) to (42);
		\draw (42) to (43);
		\draw (7) to (42);
		\draw (28) to (62);
		\draw (29) to (62);
		\draw (62) to (39);
		\draw (62) to (40);
		\draw [fill=black!10] (120.center)
		to (107.center)
		to (99.center)
		to (98.center)
		to (96.center)
		to (95.center)
		to (100.center)
		to cycle;
		\draw [fill=black!10] (69.center)
		to (70.center)
		to (71.center)
		to (72.center)
		to (73.center)
		to (74.center)
		to cycle;
		\draw [fill=black!10] (65.center)
		to (66.center)
		to (67.center)
		to (68.center)
		to (63.center)
		to (64.center)
		to cycle;
		\draw (63) to (64);
		\draw (64) to (65);
		\draw (65) to (66);
		\draw (66) to (67);
		\draw (67) to (68);
		\draw (68) to (63);
		\draw (69) to (70);
		\draw (70) to (71);
		\draw (71) to (72);
		\draw (72) to (73);
		\draw (73) to (74);
		\draw (74) to (69);
		\draw (63) to (76);
		\draw (76) to (97);
		\draw (97) to (98);
		\draw (98) to (99);
		\draw (99) to (107);
		\draw (107) to (108);
		\draw (108) to (69);
		\draw [style=alternatepath] (68) to (77);
		\draw [style=alternatepath] (77) to (87);
		\draw [style=alternatepath] (87) to (95);
		\draw [style=alternatepath] (95) to (100);
		\draw [style=alternatepath] (100) to (106);
		\draw [style=alternatepath] (106) to (70);
		\draw (67) to (79);
		\draw (79) to (88);
		\draw (88) to (93);
		\draw (93) to (101);
		\draw (101) to (105);
		\draw (105) to (71);
		\draw [style=alternatepath] (66) to (80);
		\draw (80) to (89);
		\draw (89) to (92);
		\draw [style=alternatepath] (92) to (102);
		\draw [style=alternatepath] (102) to (103);
		\draw [style=alternatepath] (103) to (110);
		\draw [style=alternatepath] (110) to (72);
		\draw [style=lightedge] (70) to (123);
		\draw [style=thickdashed] (69) to (123);
		\draw [style=lightedge] (123) to (74);
		\draw [style=lightedge] (123) to (73);
		\draw [style=lightedge] (123) to (72);
		\draw [style=thickdashed] (123) to (71);
		\draw [style=thickdashed] (63) to (122);
		\draw [style=lightedge] (64) to (122);
		\draw [style=lightedge] (122) to (68);
		\draw [style=thickdashed] (122) to (67);
		\draw [style=lightedge] (122) to (66);
		\draw [style=lightedge] (122) to (65);
		\draw (98) to (96);
		\draw (96) to (95);
		\draw (100) to (120);
		\draw (107) to (120);
		\draw [style=thickdashed] (98) to (124);
		\draw [style=lightedge] (99) to (124);
		\draw [style=thickdashed] (124) to (107);
		\draw [style=lightedge] (124) to (120);
		\draw [style=lightedge] (124) to (100);
		\draw [style=lightedge] (124) to (95);
		\draw [style=lightedge] (96) to (124);
		\draw [style=red] (63) to (76);
		\draw [style=red] (76) to (97);
		\draw [style=red] (97) to (98);
		\draw (98) to (99);
		\draw (99) to (107);
		\draw [style=red] (107) to (108);
		\draw [style=red] (108) to (69);
		\draw [style=red] (67) to (79);
		\draw [style=red] (79) to (88);
		\draw [style=red] (88) to (93);
		\draw [style=red] (93) to (101);
		\draw [style=red] (101) to (105);
		\draw [style=red] (105) to (71);
		\draw (76) to (68);
		\draw (68) to (78);
		\draw (78) to (77);
		\draw (97) to (78);
		\draw (76) to (78);
		\draw (78) to (96);
		\draw (97) to (96);
		\draw (96) to (77);
		\draw (96) to (87);
		\draw (108) to (120);
		\draw (120) to (106);
		\draw (108) to (106);
		\draw (108) to (70);
		\draw (107) to (119);
		\draw (119) to (108);
		\draw (119) to (118);
		\draw (118) to (69);
		\draw (118) to (117);
		\draw (117) to (74);
		\draw (117) to (116);
		\draw (99) to (119);
		\draw (108) to (118);
		\draw (69) to (117);
		\draw (74) to (116);
		\draw (116) to (115);
		\draw (116) to (73);
		\draw (73) to (115);
		\draw (115) to (114);
		\draw (73) to (114);
		\draw (115) to (113);
		\draw (114) to (113);
		\draw (113) to (112);
		\draw (111) to (113);
		\draw (114) to (111);
		\draw (72) to (114);
		\draw (72) to (111);
		\draw (110) to (111);
		\draw (110) to (112);
		\draw (111) to (112);
		\draw (103) to (112);
		\draw (121) to (99);
		\draw (121) to (98);
		\draw (121) to (97);
		\draw (121) to (86);
		\draw (86) to (97);
		\draw (86) to (76);
		\draw (86) to (75);
		\draw (75) to (63);
		\draw (86) to (63);
		\draw (75) to (64);
		\draw (75) to (85);
		\draw (85) to (64);
		\draw (85) to (84);
		\draw (84) to (83);
		\draw [bend left=15] (85) to (121);
		\draw (84) to (64);
		\draw (84) to (65);
		\draw (83) to (65);
		\draw (83) to (82);
		\draw (83) to (66);
		\draw (82) to (81);
		\draw (82) to (66);
		\draw (66) to (81);
		\draw (81) to (80);
		\draw (81) to (90);
		\draw [style=alternatepath] (80) to (90);
		\draw (90) to (89);
		\draw [style=alternatepath] (90) to (91);
		\draw (89) to (91);
		\draw [style=alternatepath] (91) to (92);
		\draw [bend right] (82) to (91);
		\draw (77) to (67);
		\draw (67) to (80);
		\draw (80) to (79);
		\draw (77) to (79);
		\draw (79) to (89);
		\draw (102) to (104);
		\draw (104) to (103);
		\draw (104) to (109);
		\draw (109) to (103);
		\draw (109) to (110);
		\draw (109) to (72);
		\draw (71) to (109);
		\draw (93) to (92);
		\draw (93) to (102);
		\draw (101) to (102);
		\draw (101) to (104);
		\draw (105) to (104);
		\draw (105) to (109);
		\draw (88) to (89);
		\draw (88) to (92);
		\draw (77) to (88);
		\draw (87) to (88);
		\draw (95) to (88);
		\draw (88) to (94);
		\draw (94) to (93);
		\draw (95) to (94);
		\draw (94) to (100);
		\draw (94) to (101);
		\draw (100) to (101);
		\draw (100) to (105);
		\draw (105) to (106);
		\draw (70) to (105);
		\draw (91) to (125);
		\draw (92) to (125);
		\draw (125) to (102);
		\draw (125) to (103);
		\draw [bend right=15] (56) to (58);
		\draw [bend right=15] (119) to (121);
		\draw [style=shaded] (183) to (170);
		\draw [style=shaded] (170) to (162);
		\draw [style=shaded] (162) to (161);
		\draw [style=shaded] (161) to (159);
		\draw [style=shaded] (159) to (158);
		\draw [style=shaded] (158) to (163);
		\draw [style=shaded] (163) to (183);
		\draw [style=shaded] (132) to (133);
		\draw [style=shaded] (133) to (134);
		\draw [style=shaded] (134) to (135);
		\draw [style=shaded] (135) to (136);
		\draw [style=shaded] (136) to (137);
		\draw [style=shaded] (137) to (132);
		\draw [style=shaded] (128) to (129);
		\draw [style=shaded] (129) to (130);
		\draw [style=shaded] (130) to (131);
		\draw [style=shaded] (131) to (126);
		\draw [style=shaded] (126) to (127);
		\draw [style=shaded] (127) to (128);
		\draw (126) to (127);
		\draw (127) to (128);
		\draw (128) to (129);
		\draw (129) to (130);
		\draw (130) to (131);
		\draw (131) to (126);
		\draw (132) to (133);
		\draw (133) to (134);
		\draw (134) to (135);
		\draw (135) to (136);
		\draw (136) to (137);
		\draw (137) to (132);
		\draw (126) to (139);
		\draw (139) to (160);
		\draw (160) to (161);
		\draw (161) to (162);
		\draw (162) to (170);
		\draw (170) to (171);
		\draw (171) to (132);
		\draw [style=alternatepath] (131) to (140);
		\draw [style=alternatepath] (140) to (150);
		\draw [style=alternatepath] (150) to (158);
		\draw [style=alternatepath] (158) to (163);
		\draw [style=alternatepath] (163) to (169);
		\draw [style=alternatepath] (169) to (133);
		\draw (130) to (142);
		\draw (142) to (151);
		\draw (151) to (156);
		\draw (156) to (164);
		\draw (164) to (168);
		\draw (168) to (134);
		\draw [style=alternatepath] (129) to (143);
		\draw (143) to (152);
		\draw (152) to (155);
		\draw [style=alternatepath] (155) to (165);
		\draw [style=alternatepath] (165) to (166);
		\draw [style=alternatepath] (166) to (173);
		\draw [style=alternatepath] (173) to (135);
		\draw (161) to (159);
		\draw (159) to (158);
		\draw (163) to (183);
		\draw (170) to (183);
		\draw [style=red] (126) to (139);
		\draw [style=red] (139) to (160);
		\draw [style=red] (160) to (161);
		\draw (161) to (162);
		\draw (162) to (170);
		\draw [style=red] (170) to (171);
		\draw [style=red] (171) to (132);
		\draw [style=red] (130) to (142);
		\draw [style=red] (142) to (151);
		\draw [style=red] (151) to (156);
		\draw [style=red] (156) to (164);
		\draw [style=red] (164) to (168);
		\draw [style=red] (168) to (134);
		\draw (139) to (131);
		\draw (131) to (141);
		\draw (141) to (140);
		\draw (160) to (141);
		\draw (139) to (141);
		\draw (141) to (159);
		\draw (160) to (159);
		\draw (159) to (140);
		\draw (159) to (150);
		\draw (171) to (183);
		\draw (183) to (169);
		\draw (171) to (169);
		\draw (171) to (133);
		\draw (170) to (182);
		\draw (182) to (171);
		\draw (182) to (181);
		\draw (181) to (132);
		\draw (181) to (180);
		\draw (180) to (137);
		\draw (180) to (179);
		\draw (162) to (182);
		\draw (171) to (181);
		\draw (132) to (180);
		\draw (137) to (179);
		\draw (179) to (178);
		\draw (179) to (136);
		\draw (136) to (178);
		\draw (178) to (177);
		\draw (136) to (177);
		\draw (178) to (176);
		\draw (177) to (176);
		\draw (176) to (175);
		\draw (174) to (176);
		\draw (177) to (174);
		\draw (135) to (177);
		\draw (135) to (174);
		\draw (173) to (174);
		\draw (173) to (175);
		\draw (174) to (175);
		\draw (166) to (175);
		\draw (184) to (162);
		\draw (184) to (161);
		\draw (184) to (160);
		\draw (184) to (149);
		\draw (149) to (160);
		\draw (149) to (139);
		\draw (149) to (138);
		\draw (138) to (126);
		\draw (149) to (126);
		\draw (138) to (127);
		\draw (138) to (148);
		\draw (148) to (127);
		\draw (148) to (147);
		\draw (147) to (146);
		\draw [bend left=15] (148) to (184);
		\draw (147) to (127);
		\draw (147) to (128);
		\draw (146) to (128);
		\draw (146) to (145);
		\draw (146) to (129);
		\draw (145) to (144);
		\draw (145) to (129);
		\draw (129) to (144);
		\draw (144) to (143);
		\draw (144) to (153);
		\draw [style=alternatepath] (143) to (153);
		\draw (153) to (152);
		\draw [style=alternatepath] (153) to (154);
		\draw (152) to (154);
		\draw [style=alternatepath] (154) to (155);
		\draw [bend right] (145) to (154);
		\draw (140) to (130);
		\draw (130) to (143);
		\draw (143) to (142);
		\draw (140) to (142);
		\draw (142) to (152);
		\draw (165) to (167);
		\draw (167) to (166);
		\draw (167) to (172);
		\draw (172) to (166);
		\draw (172) to (173);
		\draw (172) to (135);
		\draw (134) to (172);
		\draw (156) to (155);
		\draw (156) to (165);
		\draw (164) to (165);
		\draw (164) to (167);
		\draw (168) to (167);
		\draw (168) to (172);
		\draw (151) to (152);
		\draw (151) to (155);
		\draw (140) to (151);
		\draw (150) to (151);
		\draw (158) to (151);
		\draw (151) to (157);
		\draw (157) to (156);
		\draw (158) to (157);
		\draw (157) to (163);
		\draw (157) to (164);
		\draw (163) to (164);
		\draw (163) to (168);
		\draw (168) to (169);
		\draw (133) to (168);
		\draw (154) to (188);
		\draw (155) to (188);
		\draw (188) to (165);
		\draw (188) to (166);
		\draw [bend right=15] (182) to (184);
	\end{pgfonlayer}
\end{tikzpicture}

%% file: 3conn_counter_example.tex
\begin{tikzpicture}
	[scale=0.6,inner sep=0.8mm, 
	blackvertex/.style={circle,thick,draw, fill=black},
	whitevertex/.style={circle,thick,draw, fill=white},
	greyvertex/.style={circle,thick,draw, fill=black!30},
	star/.style={shape=star, star points = 5, thick, draw, fill=white},
	none/.style={},
	red/.style={line width=2pt},
	dashedthick/.style={line width=2pt, dotted},
	lightedge/.style={line width=1pt, color=black!50, dotted}]
	
	\pgfdeclarelayer{bg} 
	\pgfdeclarelayer{nodelayer} 
	\pgfdeclarelayer{edgelayer}  
	\pgfsetlayers{bg, edgelayer, nodelayer, main} 
	
	\begin{pgfonlayer}{nodelayer}
		\node [style=whitevertex] (0) at (-4, 0) {$1$};
		\node [style=whitevertex] (1) at (4, 0) {$4$};
		\node [style=whitevertex] (2) at (0, 4) {$6$};
		\node [style=whitevertex] (3) at (-1.5, -0.5) {$2$};
		\node [style=whitevertex] (4) at (1.5, -0.5) {$3$};
		\node [style=whitevertex] (5) at (0, 1.5) {$8$};
		\node [style=whitevertex] (6) at (-1.25, 2.5) {$5$};
		\node [style=whitevertex] (7) at (1.25, 2.5) {$7$};
	\end{pgfonlayer}
	\begin{pgfonlayer}{edgelayer}
		\draw [fill=black!10] (3.center)
		to (0.center)
		to [bend right=45] (1.center)
		to (4.center)
		to cycle;
		\draw [fill=black!10] (6.center)
		to (5.center)
		to (7.center)
		to (2.center)
		to cycle;
		\draw [bend right=45] (0.center) to (1.center);
		\draw [bend left] (0) to (2);
		\draw [bend left] (2) to (1);
		\draw (0) to (3);
		\draw (3) to (4);
		\draw (4) to (1);
		\draw (5) to (3);
		\draw (5) to (4);
		\draw (6) to (5);
		\draw (5) to (7);
		\draw (6) to (2);
		\draw (2) to (7);
		\draw (0) to (5);
		\draw (5) to (1);
		\draw (0) to (6);
		\draw (7) to (1);
	\end{pgfonlayer}
\end{tikzpicture}